\documentclass[10pt,a4paper]{amsart}
\usepackage[lmargin=0.9in,rmargin=0.9in]{geometry}
\usepackage[T1]{fontenc}
\usepackage[utf8]{inputenc}
\usepackage[british]{babel}
\usepackage{mathtools}
\usepackage{amsthm}
\usepackage{lmodern}
\usepackage{calrsfs}
\usepackage{amssymb}
\usepackage{amsmath, calligra, mathrsfs}
\usepackage[mathscr]{euscript}
\usepackage{enumitem}
\usepackage{tikz-cd}
\usetikzlibrary{decorations.markings}
\usepackage{hyperref}
\usepackage[noabbrev,nameinlink]{cleveref}
\theoremstyle{plain}
\newtheorem{thm}{Theorem}[section]
\newtheorem*{thm*}{Theorem}
\newtheorem{lem}[thm]{Lemma}
\newtheorem{prop}[thm]{Proposition}
\newtheorem{cor}[thm]{Corollary}

\theoremstyle{definition}
\newtheorem{defn}[thm]{Definition}
\newtheorem{nota}[thm]{Notation}

\theoremstyle{remark}
\newtheorem{rem}[thm]{Remark}
\Crefname{thm}{Theorem}{Theorems}
\Crefname{lm}{Lemma}{Lemmata}
\Crefname{prop}{Proposition}{Propositions}
\Crefname{cor}{Corollary}{Corollaries}
\Crefname{hyp}{Hypothesis}{Hypotheses}
\Crefname{q}{Question}{Questions}
\Crefname{defn}{Definition}{Definitions}
\Crefname{nota}{Notation}{Notations}
\Crefname{ex}{Example}{Examples}
\Crefname{xca}{Exercise}{Exercises}
\Crefname{rem}{Remark}{Remarks}
\Crefname{constr}{Construction}{Constructions}
\Crefname{conj}{Conjecture}{Conjecture}

\newcommand{\Z}{\mathbb{Z}}
\newcommand{\Q}{\mathbb{Q}}

\newcommand{\Acal}{\mathcal{A}}

\newcommand{\Ccal}{\mathcal{C}}
\newcommand{\Dcal}{\mathcal{D}}
\newcommand{\Ecal}{\mathcal{E}}

\newcommand{\Ical}{\mathcal{I}}

\newcommand{\Lcal}{\mathcal{L}}

\newcommand{\Scal}{\mathcal{S}}
\newcommand{\Ucal}{\mathcal{U}}

\newcommand{\id}{\textup{id}}
\newcommand{\Hom}{\textup{Hom}}
\newcommand{\End}{\textup{End}}
\newcommand{\Fin}{\mathsf{Fin}}

\newcommand{\Ab}{\mathbf{Ab}}

\newcommand{\Hbb}{\mathbb{H}} % stable homotopic 2-functors
\newcommand{\Fun}{\mathbf{Fun}} % functors

\newcommand*{\isoarrow}[1]{\arrow[#1,"\rotatebox{90}{\(\sim\)}"]}

\newcommand{\A}{\mathbf{A}} % abelian hull
\newcommand{\conn}{\mathsf{conn}}
\newcommand{\modules}{\mathsf{mod}}

\usepackage{changepage} 

\title{On the functoriality of universal abelian factorizations}
\author{Luca Terenzi}
\address{Luca Terenzi \newline
	\indent UMPA, ENS de Lyon \newline
	\indent 46 Allée d'Italie, 69007 Lyon (France)}
\email{\normalfont\href{mailto:luca.terenzi@ens-lyon.fr}{luca.terenzi@ens-lyon.fr}}

\makeatletter
\hypersetup{
	pdfauthor={\author},
	pdftitle={\@title},
	colorlinks,
	linkcolor=[rgb]{0.2,0.2,0.6},
	citecolor=[rgb]{0.2,0.6,0.2},
	urlcolor=[rgb]{0.6,0.2,0.2}}
\makeatother

\calclayout
\begin{document}

\maketitle

\begin{abstract}
	In this note, we discuss several aspects of the functoriality of universal abelian factorizations associated to representations of quivers into abelian categories. After recalling the general construction of universal abelian factorizations, we review the canonical lifting procedures for exact functors and natural transformations thereof (already studied by F. Ivorra in a slightly different axiomatic framework) and we describe how these interact with general categorical constructions; for sake of simplicity, we mostly focus on the case where the quivers and representations considered are defined by actual (additive) categories and (additive) functors. We then extend these known results in two directions which have not been explored explicitly in the existing literature: on the one side, to the setting of multi-linear functors; on the other side, to the setting of abelian fibered categories. Combining these two extensions, we are able to discuss the case of monoidal structures on abelian fibered categories; the latter case has been successfully applied in the author's construction of the tensor structure on perverse Nori motives.
\end{abstract}

\tableofcontents

\section*{Introduction}

\subsection*{Motivation and goal of the paper}

Let $\beta: \Dcal \rightarrow \Acal$ be an additive functor from an additive category $\Dcal$ to an abelian category $\Acal$. As shown in \cite{BV-P}, there exists an abelian category $\A(\beta)$ providing a universal factorization of $\beta$ of the form
\begin{equation*}
	\beta: \Dcal \xrightarrow{\pi} \A(\beta) \xrightarrow{\iota} \Acal
\end{equation*}
where $\pi: \Dcal \rightarrow \A(\beta)$ is an additive functor and $\iota: \A(\beta) \rightarrow \Acal$ is a faithful exact functor. The abelian category $\A(\beta)$ is known as the \textit{universal abelian factorization} associated to the additive functor $\beta$: one can view it as the finest possible non-full abelian subcategory of $\Acal$ keeping track of all the information coming from $\Dcal$ via the functor $\beta$. 

This categorical construction is motivated by the theory of mixed motives and crucial for the emerging theory of perverse Nori motives introduced by F. Ivorra and S. Morel in \cite{IM19}. The results of \cite{IM19} and \cite{Ter23Nori} completely describe the functoriality of the categories of perverse Nori motives with respect to Grothendieck's six functor formalism. The basic building blocks for the six functor formalism on perverse Nori motives all come from the abstract lifting properties of universal abelian factorizations with respect to exact functors and natural transformations thereof. 

The main purpose of this note is to study these abstract aspects of universal abelian factorizations systematically. Some of the main results collected here have been already obtained before, based on a different axiomatic framework, and can be found in \cite{Ara-mot,Ara-rev}, \cite{HMS17}, \cite{BVHP20}, \cite{IvPerv}, and \cite{IM19}. Our motivation to write the present note is twofold: firstly, in order to give a uniform and self-contained treatment of the previously known results; secondly, to fill some gaps in the existing literature. Several categorical constructions recurring in \cite{IM19} and \cite{Ter23Nori} fit perfectly into the general framework of the present note; in particular, certain technical proofs of \cite{Ter23Fib} become both considerably shorter and conceptually cleaner in light of the abstract results collected here. 

In view of the generality and usefulness of the concept of universal abelian factorizations, we believe that the results of the present note will find other applications within the theory of Nori motives and beyond.

\subsection*{History and previous work}

In the late 1990's, M. Nori found an elegant way to construct a candidate for the conjectural abelian category of mixed motives over a field. Nori's construction is categorical-theoretic at heart, since it relies on a general result concerning representations of quivers into the abelian category of finitely generated modules over a Noetherian commutative ring $R$. He showed that any such quiver representation 
\begin{equation*}
	\beta: \Dcal \rightarrow \modules_R^{f.g.}
\end{equation*}
admits a universal factorization of the form
\begin{equation*}
	\beta: \Dcal \xrightarrow{\pi} \A(\beta) \xrightarrow{\iota} \modules_R^{f.g.}
\end{equation*} 
where $\A(\beta)$ is an $R$-linear abelian category, $\pi$ is a quiver representation, and $\iota: \A(\beta) \rightarrow \modules_R^{f.g.}$ is a faithful exact $R$-linear functor. The universal property amounts to the fact that, for any other factorization
\begin{equation*}
	\beta: \Dcal \xrightarrow{\pi'} \Acal \xrightarrow{\iota'} \modules_R^{f.g.}
\end{equation*}
of the same form, there exists a unique faithful exact $R$-linear functor $\A(\beta) \rightarrow \Acal$ making the diagram
\begin{equation*}
	\begin{tikzcd}
		&& \A(\beta) \arrow{drr}{\iota} \arrow[dashed]{dd}{\exists!} \\
		\Dcal \arrow{urr}{\pi} \arrow{drr}{\pi'} &&&& \modules_R^{f.g.} \\
		&& \Acal \arrow{urr}{\iota'}
	\end{tikzcd}
\end{equation*}
commute. The category $\A(\beta)$ is constructed explicitly in terms of the representation $\beta$ as
\begin{equation*}
	\A(\beta) := 2-\varinjlim_{\Dcal' \in \Fin(\Dcal)} \End(\beta|_{\Dcal'})-\modules_R^{f.g.},
\end{equation*}
where $\Fin(\Dcal)$ denotes the filtered poset of all finite full subquivers of $\Dcal$ and, for each inclusion $\Dcal' \subset \Dcal''$ in $\Fin(\Dcal)$, the corresponding functor $\End(\beta|_{\Dcal'})-\modules_R^{f.g.} \rightarrow \End(\beta|_{\Dcal'})-\modules_R^{f.g.}$ in the above $2$-colimit is just the restriction-of-scalars functor induced by the obvious $R$-algebra homomorphism $\End(\beta|_{\Dcal''}) \rightarrow \End(\beta|_{\Dcal'})$.

It is clear from this description that the association $\beta \mapsto \A(\beta)$ is functorial in the following sense: given two quiver representations
\begin{equation*}
	\beta_1: \Dcal_1 \rightarrow \modules_R^{f.g.}, \quad \beta_2: \Dcal_2 \rightarrow \modules_R^{f.g.}
\end{equation*}
together with a map of quivers $\phi: \Dcal_1 \rightarrow \Dcal_2$ fitting into a commutative diagram of the form
\begin{equation*}
	\begin{tikzcd}
		\Dcal_1 \arrow{drr}{\beta_1} \arrow{rrrr}{\phi} &&&& \Dcal_2 \arrow{dll}{\beta_2} \\
		&& \modules_R^{f.g.}
	\end{tikzcd}
\end{equation*}
there exists a unique exact $R$-linear functor $\bar{\phi}: \A(\beta_1) \rightarrow \A(\beta_2)$ making the whole diagram
\begin{equation*}
	\begin{tikzcd}
		\Dcal_1 \arrow{d}{\pi_1} \arrow{rrrr}{\phi} &&&& \Dcal_2 \arrow{d}{\pi_2} \\
		\A(\beta_1) \arrow{drr}{\iota_1} \arrow[dashed]{rrrr}{\bar{\phi}} &&&& \A(\beta_2) \arrow{dll}{\iota_2} \\
		&& \modules_R^{f.g.}
	\end{tikzcd}
\end{equation*}
commute. It is also clear that the association $\phi \mapsto \bar{\phi}$ is compatible with composition of maps of quivers in the obvious way. Motivated by his applications to mixed motives, Nori also described a general recipe to construct a tensor product on the category $\A(\beta)$: the idea is to start from a pairing $- \times -: \Dcal \times \Dcal \rightarrow \Dcal$ making the diagram
\begin{equation*}
	\begin{tikzcd}
		\Dcal \times \Dcal \arrow{rr}{- \times -} \arrow{d}{\beta \times \beta} && \Dcal \arrow{d}{\beta} \\
		\modules_R^{f.g.} \times \modules_R^{f.g.} \arrow{rr}{- \otimes_R -} && \modules_R^{f.g.}
	\end{tikzcd}
\end{equation*}
commute up to natural isomorphism, and extend such a pairing along each step of the construction of $\A(\beta)$ in order to get a bi-functor $- \otimes -: \A(\beta) \times \A(\beta) \rightarrow \A(\beta)$ making the whole diagram
\begin{equation*}
	\begin{tikzcd}
		\Dcal \times \Dcal \arrow{rr}{- \times -} \arrow{d}{\pi \times \pi} && \Dcal \arrow{d}{\pi} \\
		\A(\beta) \times \A(\beta) \arrow[dashed]{rr}{- \otimes_R -} \arrow{d}{\iota \times \iota} && \A(\beta) \arrow{d}{\iota} \\
		\modules_R^{f.g.} \times \modules_R^{f.g.} \arrow{rr}{- \otimes_R -} && \modules_R^{f.g.}
	\end{tikzcd}
\end{equation*}
commute up to natural isomorphism. In order to make this idea work, one has to make sure that the representation $\beta: \Dcal \rightarrow \modules_R^{f.g.}$ takes values into flat $R$-modules and that, in addition, iterated kernels of such $R$-modules are still flat. Thus, in concrete applications, one is basically forced to assume that either $R$ is a field or $R$ is a Dedekind domain and the $R$-module $\beta(D)$ is torsion-free for all $D \in \Dcal$. See \cite[\S~II]{HMS17} for a thorough exposition of Nori's constructions and results.

Until the last decade, the range of applications of Nori's categorical construction remained limited to representations of quivers into $\modules_R^{f.g.}$, possibly with further restrictions on the ring $R$. While this setting was certainly sufficient for Nori's original applications to mixed motives, it appeared to be too rigid for possible generalization of Nori's theory of motives to a theory of motivic sheaves. The basic idea would be to mimic Nori's construction starting from a representation of some quiver into the abelian categories of constructible or perverse sheaves over algebraic varieties, which are not of the form $\modules_R^{f.g.}$ despite being Noetherian (and even Artinian, in the case of perverse sheaves). 

In the setting of ordinary constructible sheaves, Arapura introduced in \cite{Ara-mot} a clever yet notationally intricate variant of Nori's construction: roughly, one filters the category of constructible sheaves over a variety $S$ by finer and finer stratifications of $S$, considers collections of stalk functors associated to the different strata, and compares the functors resulting from different choices of stalks. 

Of course, Arapura's approach cannot be applied to the setting of perverse sheaves over general varieties. In this setting, Ivorra observed in \cite{IvPerv} that Nori's construction can be extended to representation of quivers into arbitrary finite and hom-finite abelian categories: this works perfectly for representations into categories of perverse sheaves (whereas it would not work for representations into ordinary constructible sheaves over general varieties, since these do not form Artinian categories). Furthermore, it makes it very easy to understand the compatibility of the construction with respect to morphisms of quiver representations: given two quiver representations into finite and hom-finite abelian categories
\begin{equation*}
	\beta_1: \Dcal_1 \rightarrow \Acal_1, \quad \beta_2: \Dcal_2 \rightarrow \Acal_2
\end{equation*}
together with a map of quivers $\phi: \Dcal_1 \rightarrow \Dcal_2$ and an exact functor $\Phi: \Acal_1 \rightarrow \Acal_2$ making the diagram
\begin{equation*}
	\begin{tikzcd}
		\Dcal_1 \arrow{d}{\beta_1} \arrow{rrrr}{\phi} &&&& \Dcal_2 \arrow{d}{\beta_2} \\
		\Acal_1 \arrow{rrrr}{\Phi} &&&& \Acal_2
	\end{tikzcd}
\end{equation*}
commute up to natural isomorphism, there exists an essentially unique exact functor $\bar{\phi}: \A(\beta_1) \rightarrow \A(\beta_2)$ making the whole diagram
\begin{equation*}
	\begin{tikzcd}
		\Dcal_1 \arrow{d}{\pi_1} \arrow{rrrr}{\phi} &&&& \Dcal_2 \arrow{d}{\pi_2} \\
		\A(\beta_1) \arrow{d}{\iota_1} \arrow[dashed]{rrrr}{\bar{\phi}} &&&& \A(\beta_2) \arrow{d}{\iota_2} \\
		\Acal_1 \arrow{rrrr}{\Phi} &&&& \Acal_2
	\end{tikzcd}
\end{equation*} 
commute up to natural isomorphism. It is worth stressing that Ivorra's result is more precise than what we have just asserted, since it includes the datum of a natural isomorphism $\kappa: \Phi \circ \beta_1 \xrightarrow{\sim} \beta_2 \circ \phi$ into the input and provides an explicit factorization of such a natural isomorphism as part of the output. A similar result holds on the level of natural transformations; again, in order to make the construction work, one needs to describe both the input and the output very carefully. 

It is natural to wonder whether Nori's construction of universal abelian factorizations can be further generalized to representations of quivers into arbitrary abelian categories. This question was given an affirmative answer thanks to O. Caramello's investigations around the categorical logic of Nori motives, which culminated in the paper \cite{BVCL18}; in the subsequent paper \cite{BV-P}, the same result was translated into the familiar language of categories. The construction of the universal abelian factorization of a general quiver representation $\beta: \Dcal \rightarrow \Acal$ can be described in a uniform and streamlined way as follows:
\begin{enumerate}
	\item[(Step 1)] Make the quiver $\Dcal$ into a category $\Dcal^{cat}$, so that the representation $\beta$ becomes a functor
	\begin{equation*}
		\beta^{cat}: \Dcal^{cat} \rightarrow \Acal.
	\end{equation*} 
	\item[(Step 2)] Make the category $\Dcal^{cat}$ into an additive category $\Dcal^{\oplus}$, so that the functor $\beta^{cat}$ extends to an additive functor 
	\begin{equation*}
		\beta^{\oplus}: \Dcal^{\oplus} \rightarrow \Acal.
	\end{equation*}
	\item[(Step 3)] Form the abelian hull $\A(\Dcal^{\oplus})$, so that the additive functor $\beta^{\oplus}$ induces an exact functor 
	\begin{equation*}
		\beta^+: \A(\Dcal^{\oplus}) \rightarrow \Acal.
	\end{equation*}
	\item[(Step 4)] Define the Serre quotient
	\begin{equation*}
		\A(\beta) := \A(\Dcal^{\oplus})/\ker(\beta^+).
	\end{equation*} 
\end{enumerate}
Here, the first non-trivial passage is that from the additive category $\Dcal^+$ to its \textit{abelian hull} $\A(\Dcal^+)$: the latter is an abelian category on which $\Dcal^+$ embeds fully faithfully, with the universal property that every additive functor from $\Dcal^+$ to some abelian category extends uniquely to an exact functor from the entire $\A(\Dcal^+)$. The existence of the category $\A(\Dcal^+)$ was proved by P. Freyd in the classical paper \cite{Freyd}. Defining $\A(\beta)$ as a Serre quotient as in the last passage amounts to rendering the exact functor $\beta^+: \A(\Dcal^+) \rightarrow \Acal$ a faithful exact functor in a universal way; with this explicit description, checking that the category $\A(\beta)$ has the same universal property as Nori's abelian category essentially reduces to writing down the various universal properties used in the passages above in sequence.
In the case of quiver representations arising from homological functors, a more streamlined description of the same construction was given in \cite{IM19}, based on the study of Freyd's abelian hulls of triangulated categories from \cite{NeeTri}. Around the same time, the case of representations into abelian tensor categories was explored in detail in \cite{BVHP20}. However, the compatibility of the latter construction of $\A(\beta)$ with respect to functors and natural transformations has not been explored yet in full generality. 

\subsection*{Main results}

In the present note, we rephrase and extend the functoriality results of \cite{IvPerv} and \cite{BVHP20} using the most general construction of universal abelian factorization given in \cite{BV-P}. Some of the fundamental results collected here are just obvious adaptations of those presented in \cite{IvPerv}, using slightly more streamlined notation. 

For sake of simplicity, we refrain from considering the general framework of representation of arbitrary quivers into abelian categories: all quivers and maps of quivers that we consider in this paper are defined by actual (additive) categories and (additive) functors. This choice is motivated by our application to perverse Nori motives and has two advantages: on the one hand, it allows us to avoid phrasing the main notions in the language of quivers and discussing their (often straightforward) extension to the additive level in detail; on the other hand, it allows us to study categorical notions (such as adjunctions, equivalences, and $2$-colimits) that are needed for our applications but lack a good analogue in the world of quivers.

The crucial functoriality results concern lifting of exact functors (\Cref{prop_lift-exfunct}) and natural transformations thereof (\Cref{prop_nat-lift}); combining these two fundamental lifting principles, we can easily study the compatibility of universal abelian factorizations with several common categorical notions, including adjunctions (\Cref{cor_lift-adj}) and equivalences (\Cref{cor_lift-equiv}).

Our main original contribution concern the extension of these functoriality results to the case of multi-linear functors (\Cref{prop:lift-multiex} and \Cref{prop:lift-multinat}) and to the setting of abelian fibered categories (\Cref{prop_lifting-fib} and \Cref{prop:lift-morfib}); this leads to a similar extension to the setting of monoidal structures on abelian fibered categories (\Cref{prop:lift-ETS}). 

Along the way, we also collect useful observations about the interaction between universal abelian factorizations and several common categorical constructions, such as direct products (\Cref{cor:A(D)-prod}) and filtered $2$-colimits (\Cref{lem_A-2colim}). 

\subsection*{Structure of the paper}

In \Cref{sect:rec-UAF} we review the construction of universal abelian factorizations described in \cite{BV-P} in the case of additive quiver representations; this gives us the opportunity to fix the general notation used throughout the paper. 

The main technical results are collected in \Cref{sect:lift-exfun} and \Cref{sect:lift-nat}, where we explain how to lift exact functors (\Cref{prop_lift-exfunct}) and natural transformations thereof (\Cref{prop_nat-lift}) to the level of universal abelian factorizations, respectively. In order to state our results in a precise form, it is convenient to introduce informally a suitable '$2$-category of quiver representations' where objects are single quiver representations into abelian categories while $1$-morphisms (resp. $2$-morphisms) consist of pairs of functors (resp. natural transformations) on the quiver level and on the abelian level satisfying a certain compatibility condition (\Cref{defn:1-mor_repr} and \Cref{defn:2mor-repr}). As a particular case, we show that the lifting procedure naturally respects adjunctions (\Cref{cor_lift-adj}) and equivalences (\Cref{cor_lift-equiv}).
Moreover, we are able to study the compatibility of universal abelian factorizations with diagram categories to some extent: most notably, we show that universal abelian factorizations are compatible with finite direct products of quiver representations (see \Cref{cor:A(D)-prod}). 

This opens the way to extending the main lifting results to the case of multi-exact functors and natural transformations thereof in \Cref{sect:multi}. Arguing by induction on the number of factors involved, we show that it is possible to lift both multi-exact functors (\Cref{prop:lift-multiex}) and their natural transformations (\Cref{prop:lift-multinat}) in the expected way. This represents our first main original contribution to the existing literature.

Our second main contribution concerns the setting of abelian categories fibered over a base, which is the subject of \Cref{sect:lift-fib}. In this case, an easy combination of the results of \Cref{sect:lift-exfun} and \Cref{sect:lift-nat} allows us to deduce analogous lifting results for single abelian fibered categories (\Cref{prop_lifting-fib}) as well as for exact morphisms of such (\Cref{prop:lift-morfib}).

The final \Cref{sect:lift-ETS} is devoted to the study of monoidal structures on universal abelian factorizations. Applying the general results of \Cref{sect:multi}, we show that it is possible to lift exact monoidal structures (\Cref{prop:lift-ETS}) together with associativity and commutativity constraints (\Cref{lem:lift-asso+symm} and \Cref{lem:lift-ac}).
In view of our applications, we have chosen to phrase everything in terms of the external tensor product following \cite{Ter23Fib}; however, similar results hold with respect to the usual internal tensor product (see \cite[Thm.~6.1, Thm.~8.1]{Ter23Fib} for the dictionary relating the two formulations). Due of lack of applications, we did not treat the case of unit constraints (see \cite[\S~5]{Ter23Fib}), but this could be easily included.

\subsection*{Acknowledgments}

The present note originated from an appendix of my Ph.D. thesis, written at the University of Freiburg under the supervision of Annette Huber-Klawitter. It is a pleasure to thank her for introducing me to the idea of universal abelian factorizations and for many useful discussions around the contents of this note, as well as for her constant support and encouragement.

\section{Recollections on universal abelian factorizations}\label{sect:rec-UAF}

We start by reviewing the construction of universal abelian factorizations of representations of additive categories into abelian categories as described in \cite{BV-P}, as well as the basic properties characterizing this construction. All the results obtained in the next sections will essentially be formal consequence of these properties. 

\subsection{Review of Freyd's abelian hull}

In the first place, we review the main properties of Freyd's construction of abelian hulls from \cite{Freyd}. Recall that, for any category $\Ccal$, a presheaf $F \in \Fun(\Ccal^{op},\Ab)$ is called \textit{finitely presented} if it can be written as a cokernel of representable presheaves.

\begin{nota}
	Let $\Dcal$ be an additive category.
	\begin{itemize}
		\item We define $\mathbf{R}(\Dcal)$ as the full additive subcategory of $\Fun(\Dcal^{op},\Ab)$ consisting of all finitely presented presheaves. 
		\item Dually, we define $\mathbf{L}(\Dcal) := \mathbf{R}(\Dcal^{op})^{op}$.
	\end{itemize}
\end{nota}

Note that any additive category $\Dcal$ admits two canonical fully faithful additive embeddings
\begin{equation*}
	\Dcal \hookrightarrow \mathbf{R}(\Dcal), \qquad \Dcal \hookrightarrow \mathbf{L}(\Dcal),
\end{equation*} 
the first one induced by the Yoneda embedding $\Dcal \hookrightarrow \Fun(\Dcal^{op},\Ab)$ and the second one arising from the embedding $\Dcal^{op} \hookrightarrow \mathbf{R}(\Dcal^{op})$. Combining the two constructions, we also obtain canonical fully faithful additive embeddings
\begin{equation*}
	\Dcal \hookrightarrow \mathbf{L}(\mathbf{R}(\Dcal)), \qquad \Dcal \hookrightarrow \mathbf{R}(\mathbf{L}(\Dcal)).
\end{equation*}
With this notation, we can summarize Freyd's results as follows:

\begin{thm}[P. Freyd]\label{thm_A(D)}
	Let $\Dcal$ be an additive category. Then the following statements hold:
	\begin{enumerate}
		\item Define the additive category 
		\begin{equation}\label{A(D)}
			\A(\Dcal) := \mathbf{L}(\mathbf{R}(\Dcal)).
		\end{equation} 
	    Then:
		\begin{enumerate}
			\item[(i)] The category $\A(\Dcal)$ is abelian.
			\item[(ii)] There exists a canonical additive equivalence of categories
			\begin{equation}\label{eq:LRRL}
				\A(\Dcal) := \mathbf{L}(\mathbf{R}(\Dcal)) = \mathbf{R}(\mathbf{L}(\Dcal))
			\end{equation}  
		    making the diagram
			\begin{equation}\label{dia:LRRL}
				\begin{tikzcd}
					& \Dcal \arrow{dl} \arrow{dr} \\
					\mathbf{L}(\mathbf{R}(\Dcal)) \arrow[equal,dashed]{rr} && \mathbf{R}(\mathbf{L}(\Dcal))
				\end{tikzcd}
			\end{equation}
		    commute.
			\item[(iii)] Every object of $\Dcal$ is projective in $\A(\Dcal)$ via the canonical embedding $\Dcal \hookrightarrow \A(\Dcal)$.
		\end{enumerate} 
		\item For every abelian category $\Acal$, the following statements hold:
		\begin{enumerate}
			\item[(i)] Every additive functor $\beta: \Dcal \rightarrow \Acal$ extends uniquely to an exact functor 
		    \begin{equation*}
		    	\beta^+: \A(\Dcal) \rightarrow \Acal.
		    \end{equation*}
			\item[(ii)] Given two additive functors $\beta_1, \beta_2: \Dcal \rightarrow \Acal$, every natural transformation $\kappa: \beta_1 \rightarrow \beta_2$ extends uniquely to a natural transformation between functors $\A(\Dcal) \rightarrow \Acal$
			\begin{equation*}
				\kappa^+: \beta_1^+ \rightarrow \beta_2^+.
			\end{equation*}
		\end{enumerate}
	\end{enumerate}
\end{thm}
\begin{proof}
	\begin{enumerate}
		\item Statement (i) is proved in \cite[Thm~4.1]{Freyd}. Once one knows that $\mathbf{L}(\mathbf{R}(\Dcal))$ satisfies the universal property in (2)(i), one deduces by symmetry that $\mathbf{R}(\mathbf{L}(\Dcal))$ satisfies the same universal property, and this gives the canonical equivalence in (ii). Via this equivalence, the fully faithful embedding $\Dcal \hookrightarrow \A(\Dcal)$ factors through the usual Yoneda embedding $\mathbf{L}(\Dcal) \hookrightarrow \A(\Dcal)$. As a consequence of the Yoneda Lemma, all objects in $\mathbf{L}(\Dcal)$ are projective in $\A(\Dcal)$, and this implies (iii).
		\item Statement (i) is proved in \cite[Thm.~4.1]{Freyd}. In order to prove (ii), let $\Ical$ denote the category with two objects $1$ and $2$ and exactly one non-identity arrow $1 \rightarrow 2$. Given an abelian category $\Acal$ as in the statement, the category $\Acal^{\Ical} := \Fun(\Ical, \Acal)$ is again abelian. Moreover, for every (additive) category $\Ccal$, giving one (additive) functor $F: \Ccal \rightarrow \Acal^{\Ical}$ is the same as giving two (additive) functors $F_1, F_2 \rightarrow \Acal$ and a natural transformation $F_1 \rightarrow F_2$; if in addition $\Ccal$ is abelian then, under this correspondence, $F$ is exact if and only if $F_1$ and $F_2$ are both exact. Using this observation, we see that (ii) follows by applying (i) to the abelian category $\Acal^{\Ical}$.
	\end{enumerate}
\end{proof}

\begin{cor}\label{cor:A(Dop)}
	For every additive category $\Dcal$, there exists a canonical equivalence
	\begin{equation}\label{eq:A(Dop)}
		\A(\Dcal^{op}) = \A(\Dcal)^{op}
	\end{equation}  
    making the diagram
	\begin{equation}\label{dia:A(D)op}
		\begin{tikzcd}
			& \Dcal^{op} \arrow{dl} \arrow{dr} \\
			\A(\Dcal^{op}) \arrow[equal,dashed]{rr} && \A(\Dcal)^{op}
		\end{tikzcd}
	\end{equation}
    commute.
\end{cor}
\begin{proof}
	Using \Cref{thm_A(D)}(1)(ii), one can construct the equivalence in the statement directly as
	\begin{equation*}
		\A(\Dcal^{op}) := \mathbf{L}(\mathbf{R}(\Dcal^{op})) = \mathbf{L}(\mathbf{L}(\Dcal)^{op}) = \mathbf{R}(\mathbf{L}(\Dcal))^{op} = \A(\Dcal)^{op},
	\end{equation*}
    where the last passage is induced by the canonical equivalence \eqref{eq:LRRL}; the commutativity of the diagram \eqref{dia:A(D)op} then follows directly from that of \eqref{dia:LRRL}. Alternatively, one can observe that the abelian categories $\A(\Dcal^{op})$ and $\A(\Dcal)^{op}$ satisfy the same universal property described in \Cref{thm_A(D)}(2)(i); of course, the resulting equivalence $\A(\Dcal^{op}) = \A(\Dcal)^{op}$ coincides with the previous one.
\end{proof}

\begin{rem}
	The abelian category $\A(\Dcal)$ defined as in \eqref{A(D)} is usually called the \textit{abelian hull} of the additive category $\Dcal$. Note that this terminology is slightly misleading, in that the abelian hull $\A(\Acal)$ of an additive category $\Acal$ is usually much bigger than $\Acal$ itself! Indeed, in view of \Cref{thm_A(D)}(2)(i), the inclusion $\Acal \hookrightarrow \A(\Acal)$ is an equivalence if and only if every additive functor from $\Acal$ to an abelian category is automatically exact; by considering the functors $\Hom_{\Acal}(A,-): \Acal \rightarrow \Ab$ for $A \in \Acal$, we deduce that this is the case if and only if all objects of $\Acal$ are projective.
\end{rem}

In the next two sections we will explain how to construct functors and natural transformations on the level of universal abelian factorizations. These constructions are based on the following easy consequence of \Cref{thm_A(D)}:

\begin{cor}\label{cor:A(D)}
	Let $\Dcal_1$ and $\Dcal_2$ be two additive categories. Then:
	\begin{enumerate}
		\item Every additive functor $\phi: \Dcal_1 \rightarrow \Dcal_2$ extends uniquely to an exact functor 
		\begin{equation*}
			\phi^+: \A(\Dcal_1) \rightarrow \A(\Dcal_2).
		\end{equation*}
		\item Given two additive functors $\phi_1, \phi_2: \Dcal_1 \rightarrow \Dcal_2$, every natural transformation $\kappa: \phi_1 \rightarrow \phi_2$ extends uniquely to a natural transformation of functors $\A(\Dcal_1) \rightarrow \A(\Dcal_2)$
		\begin{equation*}
			\kappa^+: \phi_1^+ \rightarrow \phi_2^+.
		\end{equation*} 
	\end{enumerate}
\end{cor}
\begin{proof}
	\begin{enumerate}
		\item This follows by applying \Cref{thm_A(D)}(2)(i) to the composite functor $\Dcal_1 \xrightarrow{\phi} \Dcal_2 \hookrightarrow \A(\Dcal_2)$.
		\item This follows by applying \Cref{thm_A(D)}(2)(ii) to the natural transformation of additive functors $\Dcal_1 \rightarrow \A(\Dcal_2)$ induced by $\kappa$.
	\end{enumerate}
\end{proof}

\begin{rem}
	For every ring $R$ there exists an $R$-linear variant of Freyd's construction: given an $R$-linear additive category $\Dcal$, it provides an $R$-linear abelian category $\A_R(\Dcal)$ such that, for every $R$-linear abelian category $\Acal$, every $R$-linear additive functor $\Dcal \rightarrow \Acal$ extends uniquely to an $R$-linear exact functor $\A_R(\Dcal) \rightarrow \Acal$; see \cite[\S~1.4]{BV-P} for more details. 
	
	For our applications we are mainly interested in the case where $R = \Q$, which is particularly easy to treat. Indeed, for an additive category, being $\Q$-linear is a property and not an additional structure; moreover, for every $\Q$-linear additive category $\Ccal$, the categories $\mathbf{R}(\Ccal)$ and $\mathbf{L}(\Ccal)$ are easily seen to be $\Q$-linear as well. As a consequence, for every $\Q$-linear additive category $\Dcal$, the abelian category $\A(\Dcal)$ is $\Q$-linear; since it obviously satisfies the universal property of $\A_{\Q}(\Dcal)$, it is canonically isomorphic to it.
\end{rem}

\subsection{Universal abelian factorizations}

It will prove convenient to treat representations of additive quivers into abelian categories as being the objects of a certain "$2$-category of representations". For the purposes of the present section, it suffices to consider objects of this $2$-category; the notions of $1$-morphism and $2$-morphism are deferred until \Cref{sect:lift-exfun} and \Cref{sect:lift-nat}, respectively.

\begin{defn}\label{defn:add-repr}
	\begin{enumerate}
		\item We call \textit{(additive) representation} a triple $(\Ccal,\Ecal;R)$ consisting of two (additive) categories $\Ccal$, $\Ecal$ and one (additive) functor $R: \Ccal \rightarrow \Ecal$.
		\item We say that an additive representation $(\Ccal,\Ecal;R)$ is \textit{abelian} if the category $\Ecal$ is abelian.
	\end{enumerate}
\end{defn}

We are ready to define universal abelian factorizations of abelian representations following \cite[\S~1.4]{BV-P}:
\begin{defn}\label{defn:A(beta)}
	Let $(\Dcal,\Acal;\beta)$ be an abelian representation. We let $\A(\beta)$ denote the abelian category obtained as the Serre quotient of the abelian hull $\A(\Dcal)$ by the Serre subcategory $\ker\left\{\beta^+: \A(\Dcal) \rightarrow \Acal\right\}$, and we call it the \textit{universal abelian factorization} of $\beta$. 
\end{defn}

\begin{nota}
	Given an abelian representation $(\Dcal,\Acal;\beta)$, we will often use the following notation:
	\begin{itemize}
		\item We let $\pi$ denote the canonical quotient functor $\A(\Dcal) \rightarrow \A(\beta)$, and we also write $\pi$ for the composite additive functor $\Dcal \rightarrow \A(\Dcal) \rightarrow \A(\beta)$.
		\item We let $\iota$ denote the induced faithful exact functor $\A(\beta) \hookrightarrow \Acal$. 
	\end{itemize}
	Whenever several abelian representations are involved, we adapt the above notation in an obvious manner.
\end{nota}

The following basic consequence of \Cref{thm_A(D)}(2)(i) explains why universal abelian factorizations deserve such a name:

\begin{prop}\label{prop:A(beta)}
	Let $(\Dcal,\Acal;\beta)$ be an abelian representation, with universal abelian factorization
	\begin{equation*}
		\Dcal \xrightarrow{\pi} \A(\beta) \xrightarrow{\iota} \Acal.
	\end{equation*}
    Moreover, let $\Acal'$ be another abelian category endowed with a faithful exact functor $\Phi: \Acal' \rightarrow \Acal$, and suppose that there exists a representation $\phi: \Dcal \rightarrow \Acal'$ making the diagram
    \begin{equation*}
    	\begin{tikzcd}
    		&& \Dcal \arrow{dll}{\phi} \arrow{drr}{\beta} \\
    		\Acal' \arrow{rrrr}{\Phi} &&&& \Acal
    	\end{tikzcd}
    \end{equation*}
    commute. Then there exists a unique faithful exact functor 
    \begin{equation}\label{A(beta):univ}
    	\A(\beta) \rightarrow \Acal'
    \end{equation} 
    making the whole diagram
    \begin{equation}\label{dia:A(beta)-univ}
    	\begin{tikzcd}
    		&& \Dcal \arrow[bend right]{ddll}{\phi} \arrow[bend left]{ddrr}{\beta} \arrow{d}{\pi} \\
    		&& \A(\beta) \arrow{drr}{\iota} \arrow[dashed]{dll} \\
    		\Acal' \arrow{rrrr}{\Phi} &&&& \Acal
    	\end{tikzcd}
    \end{equation}
    commute.
\end{prop}
\begin{proof}
	By \Cref{thm_A(D)}(2)(i), the abelian representation $\phi: \Dcal \rightarrow \Acal'$ extends canonically to an exact functor
	\begin{equation*}
		\phi^+: \A(\Dcal) \rightarrow \Acal'.
	\end{equation*}
	By the universal property of Serre quotients, the latter factors though $\A(\beta)$ if and only if, for an object $A \in \A(\Dcal)$, the implication
	\begin{equation*}
		\phi^+(A) = 0 \implies \beta^+(A) = 0
	\end{equation*}
    holds; if this is the case, the induced exact functor \eqref{A(beta):univ} is uniquely determined by this property and makes the whole diagram \eqref{dia:A(beta)-univ} commute. But we have the chain of implications
    \begin{equation*}
    	\phi^+(A) = 0 \implies \Phi(\phi^+(A)) = 0 \iff \beta^+(A) = 0,
    \end{equation*}
    where the second passage follows from the equality $\Phi \circ \phi^+ = \beta^+$ (the latter, in turn, follows from the equality $\Phi \circ \phi = \beta$, which holds by hypothesis). This proves the claim. 
\end{proof}

\begin{cor}
	Let $(\Dcal,\Acal;\beta)$ be an abelian representation, and let $\beta^{op}: \Dcal^{op} \rightarrow \Acal^{op}$ denote the functor induced by $\beta$ on the opposite categories. Then there exists a canonical additive equivalence 
	\begin{equation*}
		\A(\beta^{op}) = \A(\beta)^{op}
	\end{equation*}
    making the diagram
    \begin{equation}\label{dia:A(beta-op)}
    	\begin{tikzcd}
    		& \Dcal^{op} \arrow{dl} \arrow{dr} \\
    		\A(\Dcal^{op}) \arrow[equal]{rr} \arrow{d}{\pi^{op}} && \A(\Dcal)^{op} \arrow{d}{\pi^{op}} \\
    		\A(\beta^{op}) \arrow{dr}{\iota^{op}} \arrow[equal,dashed]{rr} && \A(\beta)^{op} \arrow{dl}{\iota^{op}} \\
    		& \Acal^{op}
    	\end{tikzcd}
    \end{equation}
    commute.
\end{cor}
\begin{proof}
	This is completely analogous to \Cref{cor:A(D)-prod}. Namely, one can construct the equivalence in the statement as
	\begin{equation*}
		\A(\beta^{op}) := \A(\Dcal^{op})/\ker(\beta^{op,+}) = \A(\Dcal)^{op}/\ker(\beta^+)^{op} =: \A(\beta)^{op}
	\end{equation*}
    where the middle equivalence is induced by \eqref{eq:A(Dop)}; the commutativity of the diagram \eqref{dia:A(beta-op)} than follows directly from that of \eqref{dia:A(D)op}. Alternatively, one can observe that the abelian categories $\A(\beta^{op})$ and $\A(\beta)^{op}$ satisfy the same universal property described in \Cref{prop:A(beta)}; of course, the resulting equivalence $\A(\beta^{op}) = \A(\beta)^{op}$ coincides with the previous one.
\end{proof}

\begin{rem}
	The result of \Cref{prop:A(beta)} seems to suggest a possible approach to recover $\A(\beta)$ directly as a subcategory of $\Acal$. Namely, consider the collection $\mathfrak{A}$ of all those additive subcategories $\Acal'$ of $\Acal$ which satisfy the following properties:
	\begin{enumerate}
		\item[(i)] $\Acal'$ contains $\beta(\Dcal)$;
		\item[(ii)] $\Acal'$ is abelian;
		\item[(iii)] the inclusion functor $\Acal' \subset \Acal$ is exact.
	\end{enumerate}
	The collection $\mathfrak{A}$ is non-empty, since it contains $\Acal$ itself. If one could order $\mathfrak{A}$ by inclusion, then \Cref{prop:A(beta)} would roughly state that the subcategory $\iota(\A(\beta))$ of $\Acal$ is the minimal object of $\mathfrak{A}$; one would thus be tempted to define $\A(\beta)$ to be the intersection of all members of $\mathfrak{A}$. One issue with this approach is that, since the subcategories $\Acal'$ of $\Acal$ above are not assumed to be full, it becomes somewhat problematic to check the stability of properties (ii) and (iii) above under intersection: for example, the kernel of the same morphism could be represented by distinct objects of $\Acal$ in two different members of $\mathfrak{A}$. This issue cannot be removed by imposing that each member of $\mathfrak{A}$ be stable under isomorphisms of its objects in $\Acal$, since this would force one to automatically include all automorphisms of each object in the image of $\beta$.
\end{rem}

Before moving further, as a quick sanity check, let us note that the construction of universal abelian factorizations is idempotent in a precise sense:

\begin{lem}\label{lem:UAF-idem}
	Let $(\Dcal,\Acal;\beta)$ be an abelian representation, with associated factorization
	\begin{equation*}
		\beta: \Dcal \xrightarrow{\pi} \A(\beta) \xrightarrow{\iota} \Acal.
	\end{equation*}
	Regard the triple $(\Dcal,\A(\beta);\pi)$ as another abelian representation, and let
	\begin{equation*}
		\pi: \Dcal \xrightarrow{\rho} \A(\pi) \xrightarrow{\nu} \A(\beta)
	\end{equation*}
	denote the corresponding factorization. Then the canonical faithful exact functor $\nu: \A(\pi) \rightarrow \A(\beta)$ is in fact an equivalence.
\end{lem}
\begin{proof}
	It suffices to show that the two quotient functors
	\begin{equation*}
		\pi: \A(\Dcal) \rightarrow \A(\beta), \qquad \rho: \A(\Dcal) \rightarrow \A(\pi)
	\end{equation*}
	have the same kernel. But, for every object $A \in \A(\Dcal)$, we have the implications
	\begin{equation*}
		\pi(A) = 0 \iff \nu(\rho(A)) = 0 \iff \rho(A) = 0,
	\end{equation*}
	where the second passage follows from the conservativity of $\nu$. This proves the claim.
\end{proof}

\section{Lifting exact functors}\label{sect:lift-exfun}

In this section we explain how to define exact functors between universal abelian factorizations. This is essentially an adaptation of the results of \cite{IvPerv} to the setting of \cite{BV-P}; similar results are discussed in \cite[\S~1]{IM19}. We firstly prove the analogue of \cite[Prop.~6.3]{IvPerv} in a quite rigid form and then explain how to make it more flexible; the latter form turns out to be a useful recognition principle in concrete applications.

\subsection{Main result}

We start by proving the general lifting result for exact functors. In order to formulate it in a compact form, it is convenient to introduce the following notion of $1$-morphism of (additive) representations:

\begin{defn}\label{defn:1-mor_repr}
	For $i = 1,2$, let $(\Ccal_i,\Ecal_i;R_i)$ be a representation in the sense of \Cref{defn:add-repr}. 
	\begin{enumerate}
		\item We call \textit{$1$-morphism of representations} from $(\Ccal_1,\Ecal_1;R_1)$ to $(\Ccal_2,\Ecal_2;R_2)$ a triple $\alpha = (F,G;\kappa)$ consisting of two (additive) functors
		\begin{equation*}
			F: \Ccal_1 \rightarrow \Ccal_2, \qquad G: \Ecal_1 \rightarrow \Ecal_2
		\end{equation*}  
		and a natural isomorphism of functors $\Ccal_1 \rightarrow \Ecal_2$
		\begin{equation*}
			\kappa: G \circ R_1 \xrightarrow{\sim} R_2 \circ F.
		\end{equation*}  
		We write it as $\alpha: (\Ccal_1,\Ecal_1;R_1) \rightarrow (\Ccal_2,\Ecal_2;R_2)$ and we often depict it as a diagram of the form
		\begin{equation*}
			\begin{tikzcd}
				\Ccal_1 \arrow{rr}{F} \arrow{d}{R_1} && \Ccal_2 \arrow{d}{R_2} \\
				\Ecal_1 \arrow{rr}{G} && \Ecal_2
			\end{tikzcd}
		\end{equation*}
		where, for notational simplicity, we do not write the natural isomorphism $\kappa$ explicitly.
		\item In case the representations $(\Ccal_1,\Ecal_1;R_1)$ and $(\Ccal_2,\Ecal_2;R_2)$ are abelian, we say that a $1$-morphism $\alpha = (F,G;\kappa): (\Ccal_1,\Ecal_1;R_1) \rightarrow (\Ccal_2,\Ecal_2;R_2)$ is \textit{exact} if the additive functor $G: \Ecal_1 \rightarrow \Ecal_2$ is exact.
	\end{enumerate}
\end{defn}

\begin{rem}
	Let us stress that, in particular, given two (additive) categories $\Ccal$ and $\Ecal$ and two (additive) functors $R_1, R_2: \Ccal \rightarrow \Ecal$, $1$-morphisms of representations of the form $(\id_{\Ccal},\id_{\Ecal};\kappa): (\Ccal,\Ecal;R_1) \rightarrow (\Ccal,\Ecal;R_2)$ correspond to natural isomorphisms $\kappa: R_1 \xrightarrow{\sim} R_2$ rather than just to natural transformations.
\end{rem}

\begin{rem}\label{nota:comp-1mor}
	There exists a natural notion of composition for $1$-morphisms of representations: given two $1$-morphisms of representations $\alpha_1 = (F_1,G_1;\kappa_1): (\Ccal_1,\Ecal_1;R_1) \rightarrow (\Ccal_2,\Ecal_2;R_2)$ and $\alpha_2 = (F_2,G_2;\kappa_2): (\Ccal_2,\Ecal_2;R_2) \rightarrow (\Ccal_3,\Ecal_3;R_3)$, we define the \textit{composite $1$-morphism} of representations as the triple
	\begin{equation*}
		(F_2 \circ F_1,G_2 \circ G_1; \kappa_2 \star \kappa_1): (\Ccal_1,\Ecal_1;R_1) \rightarrow (\Ccal_2,\Ecal_1;R_2),
	\end{equation*}
	where $\kappa_2 \star \kappa_1$ denotes the natural isomorphism of functors $\Ccal_1 \rightarrow \Ecal_3$ defined by the formula
	\begin{equation*}
		(\kappa_2 \star \kappa_1)(C_1) := \kappa_2(F_1(C_1)) \circ G_2(\kappa_1(C_1)): G_2 G_1 R_1(C_1) \xrightarrow{\sim} R_3 F_2 F_1(C_1).
	\end{equation*}
	We write it as $\alpha_2 \circ \alpha_1$, and we often depict it as a diagram of the form
	\begin{equation*}
		\begin{tikzcd}
			\Ccal_1 \arrow{rr}{F_1} \arrow{d}{R_1} && \Ccal_2 \arrow{rr}{F_2} \arrow{d}{R_2} && \Ccal_3 \arrow{d}{R_3}  \\
			\Ecal_1 \arrow{rr}{G_1} && \Ecal_2 \arrow{rr}{G_2} && \Ecal_3
		\end{tikzcd}
	\end{equation*}
    where, for notational simplicity, we do not write the natural isomorphisms $\kappa_1$ and $\kappa_2$ explicitly.
\end{rem}

\begin{rem}
	Clearly, the composite of two exact $1$-morphisms of representations is exact as well.
\end{rem}

We can now state our first main lifting result (analogous to \cite[Prop.~6.3]{IvPerv}) as follows:

\begin{prop}\label{prop_lift-exfunct}
	For $i = 1,2$ let $(\Dcal_i,\Acal_i;\beta_i)$ be an abelian representation. Suppose that we are given an exact morphism of abelian representations $\alpha = (\phi,\Phi;\kappa): (\Dcal_1,\Acal_1;\beta_1) \rightarrow (\Dcal_2,\Acal_2;\beta_2)$. Then:
	\begin{enumerate}
		\item There exists a unique exact functor $\overline{\phi}: \A(\beta_1) \rightarrow \A(\beta_2)$ satisfying the equality of additive functors $\Dcal_1 \rightarrow \A(\beta_2)$
		\begin{equation}\label{equal:pi2-phi=barphi-pi1}
			\pi_2 \circ \phi = \overline{\phi} \circ \pi_1.
		\end{equation}
		We let $\hat{\alpha}$ denote the exact morphism of abelian representations
		\begin{equation*}
			(\phi,\overline{\phi}; \id): (\Dcal_1,\A(\beta_1);\pi_1) \rightarrow (\Dcal_2,\A(\beta_2);\pi_2).
		\end{equation*} 
		\item There exists a unique natural isomorphism of functors $\A(\beta_1) \rightarrow \Acal_2$
		\begin{equation}\label{tilde_theta}
			\tilde{\kappa}: \Phi \circ \iota_1 \xrightarrow{\sim} \iota_2 \circ \overline{\phi}
		\end{equation}  
		making the diagram of functors $\Dcal_1 \rightarrow \Acal_2$
		\begin{equation}\label{dia-D1A2}
			\begin{tikzcd}
				\Phi \circ \beta_1 \arrow[equal]{d} \arrow{rr}{\kappa} && \beta_2 \circ \phi \arrow[equal]{d} \\
				\Phi \circ \iota_1 \circ \pi_1 \arrow[dashed]{r}{\tilde{\kappa}} & \iota_2 \circ \overline{\phi} \circ \pi_1 \arrow[equal]{r} & \iota_2 \circ \pi_2 \circ \phi
			\end{tikzcd}
		\end{equation}
		commute. We let $\tilde{\alpha}$ denote the exact morphism of abelian representations 
		\begin{equation*}
			(\overline{\phi},\Phi;\tilde{\kappa}): (\A(\beta_1),\Acal_1;\iota_1) \rightarrow (\A(\beta_2),\Acal_2;\iota_2).
		\end{equation*}
	\end{enumerate}
\end{prop}
\begin{proof}
	Since the abelian category $\A(\beta_1)$ is generated under (co)kernels by the image of $\Dcal_1$, an exact functor $\overline{\phi}$ as in (1) is uniquely determined by the equality \eqref{equal:pi2-phi=barphi-pi1}, provided it exists. For the same reason, a natural isomorphism $\tilde{\kappa}$ as in (2) is uniquely determined by the commutativity of the diagram \eqref{dia-D1A2}, provided it exists. In order to show the existence of $\overline{\phi}$ and $\tilde{\kappa}$, we argue as follows:
	\begin{enumerate}
		\item Applying \Cref{cor:A(D)}, we obtain a canonical exact functor 
		\begin{equation*}
			\phi^+: \A(\Dcal_1) \rightarrow \A(\Dcal_2).
		\end{equation*}
		and a canonical natural isomorphism of functors $\A(\Dcal_1) \rightarrow \Acal_2$
		\begin{equation*}
			\kappa^+: \Phi \circ \beta_1 \xrightarrow{\sim} \beta_2 \circ \phi^+.
		\end{equation*}
		By the universal property of Serre quotients, a functor $\overline{\phi}$ as in the statement exists if and only if, for an object $A \in \A(\Dcal_1)$, the implication
		\begin{equation*}
			\pi_1(A) = 0 \implies \pi_2(\phi^+(A)) = 0
		\end{equation*}
		holds. But we have the chain of implications
		\begin{equation*}
			\begin{aligned}
				\pi_1(A) = 0 \implies \iota_1(\pi_1(A)) = 0 \implies & \Phi(\iota_1(\pi_1(A))) = 0 \\
				\iff & \iota_2(\pi_2(\phi^+(A))) = 0 \implies \pi_2(\phi^+(A)) = 0,
			\end{aligned}
		\end{equation*}
		where the third passage follows from the fact that the objects $\Phi(\iota_1(\pi_1(A))) = \Phi(\beta_1^+(A))$ and $\iota_2(\pi_2(\phi(A))) = \beta_2^+(\phi(A))$ are isomorphic via $\kappa^+(A)$, while the last passage follows from the conservativity of $\iota_2$.
		\item As a consequence of the uniqueness property in \Cref{thm_A(D)}(2)(ii), we see that a natural isomorphism as in \eqref{tilde_theta} makes the diagram \eqref{dia-D1A2} commute if and only if makes the induced diagram of exact functors $\A(\Dcal_1) \rightarrow \Acal_2$
		\begin{equation*}
			\begin{tikzcd}
				\Phi \circ \beta_1^+ \arrow{rr}{\kappa} \arrow[equal]{d} && \beta_2^+ \circ \phi^+ \arrow[equal]{d} \\
				\Phi \circ \iota_1 \circ \pi_1 \arrow{r}{\tilde{\kappa}} & \iota_2 \circ \overline{\phi} \circ \pi_1 \arrow[equal]{r} & \iota_2 \circ \pi_2 \circ \phi^+ 
			\end{tikzcd}
		\end{equation*}
		commute. Since the quotient functor $\pi_1: \A(\Dcal_1) \rightarrow \A(\beta_1)$ is a bijection on objects, we deduce that the only possible definition of \eqref{tilde_theta} is precisely via the latter diagram.
	\end{enumerate}
\end{proof}

\begin{rem}
	Clearly, the result of \Cref{prop:A(beta)} is just a particular case of \Cref{prop_lift-exfunct}.
\end{rem}

\begin{rem}\label{rem_comp-lift}
	\begin{enumerate}
		\item Note that, while the existence of $\overline{\phi}$ follows from the existence of the exact functor $\Phi$ (and of the natural isomorphism $\kappa$), the actual expression of $\overline{\phi}$ only depends on $\phi$. 
		\item As a consequence of the uniqueness in the two statements of \Cref{prop_lift-exfunct}, we see that the assignments $\alpha \mapsto \hat{\alpha}$ and $\alpha \mapsto \tilde{\alpha}$ are both compatible with composition of $1$-morphisms of representations (as defined in \Cref{nota:comp-1mor}) in the following sense: Given two exact $1$-morphisms of abelian representations $\alpha_1 = (\phi_1,\Phi_1;\kappa_1): (\Dcal_1,\Acal_1;\beta_1) \rightarrow (\Dcal_2,\Acal_2;\kappa_2)$ and $\alpha_2 = (\phi_2,\Phi_2;\kappa_2): (\Dcal_2,\Acal_2;\beta_2) \rightarrow (\Dcal_3,\Acal_3;\kappa_3)$, we have the equalities 
		\begin{equation*}
			\widehat{(\alpha_2 \circ \alpha_1)} = \hat{\alpha}_2 \circ \hat{\alpha}_1, \quad \qquad \widetilde{(\alpha_2 \circ \alpha_1)} = \tilde{\alpha}_2 \circ \tilde{\alpha}_1.
		\end{equation*} 
		Explicitly, this means that we have the equality of exact functors $\A(\beta_1) \rightarrow \A(\beta_3)$
		\begin{equation*}
			\overline{\phi_2 \circ \phi_1} = \overline{\phi_2} \circ \overline{\phi_1}
		\end{equation*}  
		and the equality of natural isomorphisms of functors $\A(\beta_1) \rightarrow \Acal_3$
		\begin{equation*}
			\widetilde{\kappa_2 \star \kappa_1} = \tilde{\kappa}_2 \star \tilde{\kappa}_1.
		\end{equation*}
	\end{enumerate}
\end{rem}

Before moving further, let us make an easy observation about conservativity on universal abelian factorizations: 

\begin{lem}\label{rem_conserv-lift}
	Let $(\Dcal,\Acal;\beta)$ be an abelian representation, and let $I$ be a (possibly infinite) index set. Suppose that we are given, for every $i \in I$, an abelian representation $(\Dcal_i,\Acal_i;\beta_i)$ and an exact $1$-morphism $\alpha_i = (\phi_i,\Phi_i;\kappa_i): (\Dcal,\Acal;\beta) \rightarrow (\Dcal_i,\Acal_i;\beta_i)$.
	Then the family of exact functors
	\begin{equation}\label{fam-cons1}
		\left\{\overline{\phi}_i: \A(\beta) \rightarrow \A(\beta_i)\right\}_{i \in I}
	\end{equation}
	obtained via \Cref{prop_lift-exfunct}(1) is conservative as soon as the family of exact functors
	\begin{equation}\label{fam-cons2}
		\left\{\Phi_i: \Acal \rightarrow \Acal_i \right\}
	\end{equation}  
	is conservative.
\end{lem}
\begin{proof}
	Recall that an exact functor between abelian categories is conservative if and only if it is faithful. In our setting, for each $i \in I$ the functor $\iota_i: \A(\beta_i) \hookrightarrow \Acal_i$ is faithful and exact be construction, and thus conservative. Therefore the family \eqref{fam-cons1} is conservative if and only if the family
	\begin{equation*}
		\left\{\iota_i \circ \overline{\phi}_i: \A(\beta) \rightarrow \Acal_i\right\}_{i \in I}
	\end{equation*}
	is conservative. Since the functors $\iota_i \circ \overline{\phi}_i$ and $\Phi_i \circ \iota$ are naturally isomorphic for each $i \in I$, this holds if and only if the family
	\begin{equation*}
		\left\{\Phi_i \circ \iota: \A(\beta) \rightarrow \Acal_i \right\}_{i \in I}
	\end{equation*}
	is conservative. Since $\iota: \A(\beta) \hookrightarrow \Acal$ is conservative, this is the case as soon as the family \eqref{fam-cons2} is conservative.
\end{proof}

\subsection{Recognition principle}

It is convenient to introduce a more flexible variant of \Cref{prop_lift-exfunct} in which the strict equality \eqref{equal:pi2-phi=barphi-pi1} is allowed to be a more general natural isomorphism. This result turns out to be extremely useful in concrete applications, where it serves as a recognition principle for exact functors between universal abelian factorizations.

\begin{cor}\label{cor:lifting-flexible}
	For $i = 1,2$ let $(\Dcal_i,\Acal_i;\beta_i)$ be an abelian representation. Suppose that we are given an exact $1$-morphism of abelian representations
	\begin{equation*}
		\alpha: (\phi,\Phi;\kappa): (\Dcal_1,\Acal_1;\beta_1) \rightarrow (\Dcal_2,\Acal_2;\beta_2).
	\end{equation*}
	Moreover, suppose that we are given
	\begin{itemize}
		\item an exact functor $\psi: \A(\beta_1) \rightarrow \A(\beta_2)$,
		\item a natural isomorphism of functors $\Dcal_1 \rightarrow \A(\beta_2)$
		\begin{equation*}
			\chi: \psi \circ \pi_1 \xrightarrow{\sim} \pi_2 \circ \phi,
		\end{equation*}
		\item a natural isomorphism of functors $\A(\beta_1) \rightarrow \Acal_2$
		\begin{equation*}
			\xi: \Phi \circ \iota_1 \xrightarrow{\sim} \iota_2 \circ \psi
		\end{equation*}
	\end{itemize}
	satisfying the equality of $1$-morphisms of representations $(\Dcal_1,\Dcal_2;\phi) \rightarrow (\Acal_1,\Acal_2;\Phi)$
	\begin{equation*}
		(\iota_1,\iota_2;\xi) \circ (\pi_1,\pi_2;\chi) = (\beta_1,\beta_2;\kappa).
	\end{equation*}
	Then there exists a unique natural isomorphism of functors $\A(\beta_1) \rightarrow \A(\beta_2)$
	\begin{equation*}
		\tilde{\chi}: \psi \xrightarrow{\sim} \overline{\phi}
	\end{equation*}
	making the diagram of functors $\A(\beta_1) \rightarrow \Acal_2$
	\begin{equation}\label{dia:tilde-chi}
		\begin{tikzcd}
			\iota_2 \circ \psi \arrow[dashed]{rr}{\tilde{\chi}} \arrow{dr}{\xi} && \iota_2 \circ \overline{\phi} \arrow{dl}{\tilde{\kappa}} \\
			& \Phi \circ \iota_1
		\end{tikzcd}
	\end{equation}
	commute.
\end{cor}
\begin{proof}
	Applying \Cref{prop_lift-exfunct}(2) to the $1$-morphism of abelian representations
	\begin{equation*}
		(\phi,\psi;\chi): (\Dcal_1,\A(\beta_1);\pi_1) \rightarrow (\Dcal_2,\A(\beta_2);\pi_2),
	\end{equation*}
    we obtain a canonical natural isomorphism of functors $\A(\beta_1) \rightarrow \A(\beta_2)$
    \begin{equation*}
    	\tilde{\chi}: \psi \circ \nu_1 \xrightarrow{\sim} \nu_2 \circ \bar{\phi}
    \end{equation*}
    making the diagram of functors $\A(\beta_1) \rightarrow \A(\beta_2)$
    \begin{equation*}
    	\begin{tikzcd}
    		\Phi \circ \pi_1 \arrow[equal]{d} \arrow{rr}{\kappa} && \pi_2 \circ \phi \arrow[equal]{d} \\
    		\Phi \circ \nu_1 \circ \rho_1 \arrow{r}{\tilde{\kappa}} & \nu_2 \circ \overline{\phi} \circ \rho_1 \arrow[equal]{r} & \nu_2 \circ \rho_2 \circ \phi
    	\end{tikzcd}
    \end{equation*}
    commute. By \Cref{lem:UAF-idem}, the two faithful exact functors $\nu_1: \A(\beta_1) \rightarrow \A(\pi_1)$ and $\nu_2: \A(\beta_2) \rightarrow \A(\pi_2)$ are in fact isomorphisms of categories; we can use these isomorphisms to identify $\A(\pi_1)$ and $\A(\pi_2)$ with $\A(\beta_1)$ and $\A(\beta_2)$, respectively. Thus we can interpret $\tilde{\chi}$ simply as a natural isomorphism of functors $\A(\beta_1) \rightarrow \A(\beta_2)$
    \begin{equation*}
    	\tilde{\chi}: \psi \xrightarrow{\sim} \overline{\phi}
    \end{equation*}  
    making the diagram of functors $\Dcal_1 \rightarrow \A(\beta_2)$
    \begin{equation}\label{dia:aux-chi}
    	\begin{tikzcd}
    		\psi \circ \pi_1 \arrow{rr}{\tilde{\chi}} \arrow{dr}{\chi} && \overline{\phi} \circ \pi_1 \arrow[equal]{dl} \\
    		& \pi_2 \circ \phi
    	\end{tikzcd}
    \end{equation}
    commute. Let us show that the diagram \eqref{dia:tilde-chi} is commutative. Since all the functors appearing in the diagram are exact functors on $\A(\beta_1)$, and the latter is generated under (co)kernels by the image of $\pi_1: \Dcal_1 \rightarrow \A(\beta_1)$, it suffices to show that the diagram of functors $\Dcal_1 \rightarrow \Acal_2$ obtained by restricting \eqref{dia:tilde-chi} along $\pi_1$
    \begin{equation*}
    	\begin{tikzcd}
    		\iota_2 \circ \psi \circ \pi_1 \arrow{rr}{\tilde{\chi}} \arrow{dr}{\xi} && \iota_2 \circ \overline{\phi} \circ \pi_1 \arrow{dl}{\tilde{\kappa}} \\
    		& \Phi \circ \iota_1 \circ \pi_1
    	\end{tikzcd}
    \end{equation*}
    is commutative. But the latter coincides with the outer part of the diagram
    \begin{equation*}
    	\begin{tikzcd}
    		\iota_2 \circ \psi \circ \pi_1 \arrow{rr}{\tilde{\chi}} \arrow[bend right=50]{ddddr}{\xi} \arrow[equal]{dr} && \iota_2 \circ \overline{\phi} \circ \pi_1 \arrow[bend left=50]{ddddl}{\tilde{\kappa}} \arrow{dl}{\chi} \\
    		& \iota_2 \circ \pi_2 \circ \phi \\
    		& \beta_2^+ \circ \phi \arrow[equal]{u} \\
    		& \Phi \circ \beta_1^+ \arrow[equal]{d} \arrow{u}{\kappa} \\
    		& \Phi \circ \iota_1 \circ \pi_1
    	\end{tikzcd}
    \end{equation*} 
    where the upper piece is the commutative diagram \eqref{dia:aux-chi} above, the left-most piece is commutative by hypothesis, and the right-most piece is commutative by \Cref{prop_lift-exfunct}(2). This concludes the proof.
\end{proof}

\subsection{Compatibility with diagram categories}

We conclude this section by deriving a few remarks about the behavior of universal abelian factorizations under categorical exponentiation; as a particular case, we deduce that universal abelian factorizations are canonically compatible with finite direct products.

\begin{nota}
	For every category $\Ccal$ and every small category $\Lcal$, we let $\Ccal^{\Lcal}$ denote the functor category $\Fun(\Lcal,\Ccal)$.
\end{nota}

The basic observation is that, given an additive category $\Dcal$ and a small category $\Lcal$, the two abelian categories $\A(\Dcal^{\Lcal})$ and $\A(\Dcal)^{\Lcal}$ are related by a canonical exact functor
\begin{equation*}
	\omega_{\Lcal}: \A(\Dcal^{\Lcal}) \rightarrow \A(\Dcal)^{\Lcal}
\end{equation*}
making the diagram 
\begin{equation*}
	\begin{tikzcd}
		& \Dcal^{\Lcal} \arrow{dl} \arrow{dr} \\
		\A(\Dcal^{\Lcal}) \arrow[dashed]{rr}{\omega_{\Lcal}} && \A(\Dcal)^{\Lcal} 
	\end{tikzcd}
\end{equation*}
commutative: both the existence and the uniqueness of $\omega_{\Lcal}$ are guaranteed by \Cref{thm_A(D)}(2)(i). In the next result, we extend this picture to the case of arbitrary abelian representations of $\Dcal$.

\begin{lem}\label{lem:Lcal}
	Let $(\Dcal,\Acal;\beta)$ be an abelian representation. Given a small category $\Lcal$, let
	\begin{equation*}
		\beta^{\Lcal}: \Dcal^{\Lcal} \rightarrow \Acal^{\Lcal}
	\end{equation*}
	denote the additive functor induced by post-composition with $\beta$. Consider the two exact functors
	\begin{equation*}
		\pi^{(\Lcal)}: \A(\Dcal^{\Lcal}) \rightarrow \A(\beta^{\Lcal}), \quad \iota^{\Lcal}: \A(\beta^{\Lcal}) \hookrightarrow \Acal^{\Lcal}
	\end{equation*}
    associated to the universal abelian factorization of the representation $(\Dcal^{\Lcal},\Acal^{\Lcal};\beta^{\Lcal})$. Then the following statements hold:
    \begin{enumerate}
        \item There exists a unique faithful exact functor
        \begin{equation*}
            \omega_{\Lcal}^{\beta}: \A(\beta^{\Lcal}) \rightarrow \A(\beta)^{\Lcal}
        \end{equation*}  
        making the diagram
        \begin{equation}\label{dia:omega_Lcal}
        	\begin{tikzcd}
        		\A(\Dcal^{\Lcal}) \arrow{d}{\pi^{(\Lcal)}} \arrow{rr}{\omega_{\Lcal}} && \A(\Dcal)^{\Lcal} \arrow{d}{\pi^{\Lcal}} \\
        		\A(\beta^{\Lcal}) \arrow{dr}{\iota^{(\Lcal)}} \arrow[dashed]{rr}{\omega_{\Lcal}^{\beta}} && \A(\beta)^{\Lcal} \arrow{dl}{\iota^{\Lcal}} \\
        		& \Acal^{\Lcal}
        	\end{tikzcd}
        \end{equation}
        commute.
        \item For every object $l \in \Lcal$, consider the two functors
        \begin{equation*}
        	l_{\Dcal}^*: \Dcal^{\Lcal} \rightarrow \Dcal, \quad l_{\Acal}^*: \Acal^{\Ical} \rightarrow \Acal
        \end{equation*}
        induced by evaluation at $l$. Then:
        \begin{enumerate}
        	\item[(i)] There exists a unique exact functor
        	\begin{equation*}
        		\overline{l_{\Dcal}^*}: \A(\beta^{\Lcal}) \rightarrow \A(\beta)
        	\end{equation*}
        	satisfying the equality of functors $\Dcal^{\Lcal} \rightarrow \A(\beta)$
        	\begin{equation}\label{eq:l^*}
        		\pi \circ l_{\Dcal}^* = \overline{l_{\Dcal}^*} \circ \pi^{(\Lcal)}
        	\end{equation}
        	as well as the equality of functors $\A(\beta^{\Ical}) \rightarrow \Acal$
        	\begin{equation*}
        		\iota \circ \overline{l_{\Dcal}^*} = l_{\Acal}^* \circ \iota^{(\Lcal)}.
        	\end{equation*}
            \item[(ii)] We have the equality
            \begin{equation*}
            	\overline{l_{\Dcal}^*} = l_{\A(\beta)}^* \circ \omega_{\Lcal}^{\beta}.
            \end{equation*}
            where $l_{\A(\beta)}^*: \A(\beta)^{\Lcal} \rightarrow \A(\beta)$ denotes the analogous functor induced by evaluation at $l$.
        \end{enumerate}  
     \end{enumerate}
\end{lem}
\begin{proof}
	\begin{enumerate}
		\item This follows applying \Cref{prop:A(beta)} to the obvious commutative triangle
		\begin{equation*}
			\begin{tikzcd}
				& \Dcal^{\Lcal} \arrow{dl} \arrow{dr} \\
				\A(\beta)^{\Lcal} \arrow{rr} && \Acal^{\Lcal}.
			\end{tikzcd}
		\end{equation*}
		\item Statement (i) follows applying \Cref{prop_lift-exfunct} to the exact morphism of abelian representations $(l_{\Dcal}^*, l_{\Acal}^*;\id): (\Dcal^{\Lcal},\Acal^{\Lcal};\beta^{\Lcal}) \rightarrow (\Dcal,\Acal;\beta)$, while (ii) follows from the fact that the functor $\overline{l_{\Dcal}^*}$ is uniquely characterized by the equation \eqref{eq:l^*}.
	\end{enumerate}
\end{proof}

\begin{cor}\label{cor:A(D)-prod}
	For $i = 1,\dots,n$ let $(\Dcal_i,\Acal_i;\beta_i)$ be an abelian representation, and consider the product representation
	\begin{equation*}
		\beta_1 \times \dots \times \beta_n: \Dcal_1 \times \dots \times \Dcal_n \rightarrow \Acal_1 \times \dots \times \Acal_n, \quad (D_1,\dots,D_n) \mapsto (\beta_1(D_1),\dots,\beta_n(D_n)).
	\end{equation*}
	Then we have a canonical equivalence of abelian categories
	\begin{equation*}
		\A(\beta_1 \times \dots \times \beta_n) = \A(\beta_1) \times \dots \times \A(\beta_n).
	\end{equation*}
\end{cor}
\begin{proof}
	It suffices to prove the result for $n = 2$; the general case will then follow by induction.
	In the case $n = 2$, applying \Cref{prop_lift-exfunct} to the obvious morphisms of representations
	\begin{equation*}
		\begin{tikzcd}
			\Dcal_1 \times \Dcal_2 \arrow{r}{pr_1} \arrow{d}{\beta_1 \times \beta_2} & \Dcal_1 \arrow{d}{\beta_1} \\
			\Acal_1 \times \Acal_2 \arrow{r}{pr_1} & \Acal_1
		\end{tikzcd}
		\qquad
		\begin{tikzcd}
			\Dcal_1 \times \Dcal_2 \arrow{r}{pr_2} \arrow{d}{\beta_1 \times \beta_2} & \Dcal_2 \arrow{d}{\beta_2} \\
			\Acal_1 \times \Acal_2 \arrow{r}{pr_2} & \Acal_2
		\end{tikzcd}
	\end{equation*}
	we get canonical exact functors $\A(\beta_1 \times \beta_2) \rightarrow \A(\beta_i)$, $i = 1,2$. Together, they yield an exact functor
	\begin{equation}\label{A12_to_A1A2}
		\A(\beta_1 \times \beta_2) \rightarrow \A(\beta_1) \times \A(\beta_2)
	\end{equation}
	fitting into a commutative diagram of the form
	\begin{equation*}
		\begin{tikzcd}
			\A(\beta_1 \times \beta_2) \arrow{rr} \arrow{dr}{\iota_{1,2}} && \A(\beta_1) \times \A(\beta_2) \arrow{dl}{\iota_1 \times \iota_2} \\
			& \Acal_1 \times \Acal_2
		\end{tikzcd}
	\end{equation*}
	Conversely, applying \Cref{prop_lift-exfunct} to the obvious morphisms of representations
	\begin{equation*}
		\begin{tikzcd}
			\Dcal_1 \arrow{r}{(\id,0)} \arrow{d}{\beta_1} & \Dcal_1 \times \Dcal_2 \arrow{d}{\beta_1 \times \beta_2} \\
			\Acal_1 \arrow{r}{(\id,0)} & \Acal_1 \times \Acal_2
		\end{tikzcd}
		\qquad
		\begin{tikzcd}
			\Dcal_2 \arrow{r}{(0,\id)} \arrow{d}{\beta_2} & \Dcal_1 \times \Dcal_2 \arrow{d}{\beta_1 \times \beta_2} \\
			\Acal_2 \arrow{r}{(0,\id)} & \Acal_1 \times \Acal_2
		\end{tikzcd}
	\end{equation*}
	we get canonical exact functors $\A(\beta_i) \rightarrow \A(\beta_1 \times \beta_2)$, $i = 1,2$. Together, they yield an exact functor
	\begin{equation}\label{A1A2_to_A12}
		\A(\beta_1) \times \A(\beta_2) \rightarrow \A(\beta_1 \times \beta_2)
	\end{equation}
	fitting into a commutative diagram of the form
	\begin{equation*}
		\begin{tikzcd}
			\A(\beta_1) \times \A(\beta_2) \arrow{rr} \arrow{dr}{\iota_1 \times \iota_2} && \A(\beta_1 \times \beta_2) \arrow{dl}{\iota_{1,2}} \\
			& \Acal_1 \times \Acal_2.
		\end{tikzcd}
	\end{equation*}
	Regarding both $\A(\beta_1 \times \beta_2)$ and $\A(\beta_1) \times \A(\beta_2)$ as (not necessarily full) abelian subcategories of $\Acal_1 \times \Acal_2$, we deduce that the two functors \eqref{A12_to_A1A2} and \eqref{A1A2_to_A12} are canonically mutually quasi-inverse equivalences.
\end{proof}

\section{Lifting natural transformations}\label{sect:lift-nat}

In this section we explain how to define natural transformations between exact functors on universal abelian factorizations. Again, the main result below is essentially a reformulation of \cite[Prop.~6.4]{IvPerv} to the setting of \cite{BV-P} and a similar result appears in \cite[\S~1]{IM19}. After discussing the general result, we specialize this discussion to the cases of adjunctions and equivalences, which is inspired by \cite[\S~2.7]{IM19}.

\subsection{Main result}

Following the approach of the previous sections, we formalize our main definitions using a natural notion of $2$-morphism of (additive) representations:

\begin{defn}\label{defn:2mor-repr}
	For $i = 1,2$ let $(\Ccal_i,\Ecal_i;R_i)$ be a representation; moreover, let $\alpha = (F,G;\kappa)$ and $\alpha' = (F',G';\kappa')$ be two morphisms of representations $(\Ccal_1,\Ecal_1;R_1) \rightarrow (\Ccal_2,\Ecal_2;R_2)$.	We call \textit{$2$-morphism} from $\alpha$ to $\alpha'$ a pair $(\lambda,\Lambda)$ consisting of
	\begin{itemize}
		\item a natural transformation of functors $\Ccal_1 \rightarrow \Ccal_2$
		\begin{equation*}
			\lambda: F(C_1) \rightarrow F'(C_1),
		\end{equation*}
		\item a natural transformation of functors $\Ecal_1 \rightarrow \Ecal_2$
		\begin{equation*}
			\Lambda: G(E_1) \rightarrow G'(E_1)
		\end{equation*}
	\end{itemize}
	such that the diagram of functors $\Ccal_1 \rightarrow \Ecal_2$
	\begin{equation*}
		\begin{tikzcd}
			G R_1(C_1) \arrow{r}{\Lambda} \arrow{d}{\kappa} & G' R_1(C_1) \arrow{d}{\kappa'} \\
			R_2 F(C_1) \arrow{r}{\lambda} & R_2 F'(C_1)
		\end{tikzcd}
	\end{equation*}
	commutes. We write it as $(\lambda,\Lambda): \alpha \rightarrow \alpha'$, and we often depict it as a diagram of the form
	\begin{equation*}
		\begin{tikzcd}
			\Ccal_1 \arrow{rr}{F} \arrow{d}{R_1} && \Ccal_2 \arrow{d}{R_2} \\
			\Ecal_1 \arrow{rr}{G} && \Ecal_2 \\
			& \Downarrow \\
			\Ccal_1 \arrow{rr}{F'} \arrow{d}{R_1} && \Ccal_2 \arrow{d}{R_2} \\
			\Ecal_1 \arrow{rr}{G'} && \Ecal_2
		\end{tikzcd}
	\end{equation*}
	where, for notational simplicity, we do not write the natural transformations $\lambda$ and $\Lambda$ explicitly.
\end{defn}

\begin{rem}\label{rem:2mor-repr}
	Note that giving two (additive) functors $F,F': \Ccal_1 \rightarrow \Ecal_1$ together with a natural transformation $\lambda: F \rightarrow F'$ is the same a giving a single (additive) functor
	\begin{equation*}
		(F \xrightarrow{\lambda} F'): \Ccal_1 \rightarrow \Ccal_2^{\Ical}, \quad C_1 \rightsquigarrow (F(C_1),F'(C_1);\lambda(C_1)),
	\end{equation*}
	where $\Ical$ denotes the category with two objects and exactly one non-identity arrow between them. Similarly, giving two (additive) functors $G,G': \Ecal_1 \rightarrow \Ecal_2$ together with a natural transformation $\Lambda: G \rightarrow G'$ is the same as giving a single (additive) functor
	\begin{equation*}
		(G \xrightarrow{\Lambda} G'): \Ecal_1 \rightarrow \Ecal_2^{\Ical}, \quad E_1 \rightsquigarrow (G(E_1),G'(E_1);\Lambda(E_1)).
	\end{equation*}
	Therefore we obtain two additive representations
	\begin{equation*}
		(G \circ R_1 \xrightarrow{\Lambda} G' \circ R_1): \Ccal_1 \xrightarrow{R_1} \Ecal_1 \xrightarrow{(G \xrightarrow{\Lambda} G')} \Ecal_2^{\Ical}, \qquad (R_2 \circ F \xrightarrow{\lambda} R_2 \circ F'): \Ccal_1 \xrightarrow{(F \xrightarrow{\lambda} F')} \Ccal_2^{\Ical} \xrightarrow{R_2} \Ecal_2^{\Ical}.
	\end{equation*} 
	Moreover, if the categories $\Ecal_1$ and $\Ecal_2$ are abelian, then the functor $(G \xrightarrow{\Lambda} G')$ is exact if and only if $G$ and $G'$ are both exact. 
	By construction, saying that the pair $(\lambda,\Lambda)$ defines a $2$-morphism $\alpha \rightarrow \alpha'$ is equivalent to saying that the two natural isomorphisms
	\begin{equation*}
		\kappa: G R_1(C_1) \xrightarrow{\sim} R_2 F(C_1), \qquad \kappa': G' R_1(C_1) \xrightarrow{\sim} R_2 F'(C_1)
	\end{equation*}
	define a natural isomorphism of functors $\Ccal_1 \rightarrow \Ecal_2^{\Ical}$
	\begin{equation*}
		(\kappa,\kappa'): (G \circ R_1 \xrightarrow{\Lambda} G' \circ R_1) \rightarrow (R_2 \circ F \xrightarrow{\lambda} R_2 \circ F').
	\end{equation*}
	In this way, we obtain a $1$-morphism of representation
	\begin{equation*}
		(\id_{\Ccal_1}, \id_{\Ecal_2^{\Ical}}; (\kappa,\kappa')): (\Ccal_1,\Ecal_2^{\Ical};(G \circ R_1 \xrightarrow{\Lambda} G' \circ R_1)) \rightarrow (\Ccal_1,\Ecal_2^{\Ical};(R_2 \circ F \xrightarrow{\lambda} R_2 \circ F')).
	\end{equation*}
\end{rem} 

\begin{rem}\label{rem:comp-2mor}
	There exists a natural notion of composition for $2$-morphisms of representations: given three $1$-morphisms $\alpha_1 = (F_1, G_1; \kappa_1)$, $\alpha_2 = (F_2, G_2; \kappa_2)$ and $\alpha_3 = (F_3, G_3; \kappa_3)$ between representations $(\Ccal_1,\Ecal_1;R_1) \rightarrow (\Ccal_2,\Ecal_2;R_2)$ as well as two $2$-morphisms of representations $(\lambda_1,\Lambda_1): \alpha_1 \rightarrow \alpha_2$ and $(\lambda_2,\Lambda_2): \alpha_2 \rightarrow \alpha_3$, we define the composite $2$-morphism of representations as the triple
	\begin{equation*}
		(\lambda_2 \circ \lambda_1, \Lambda_2 \circ \Lambda_1): \alpha_1 \rightarrow \alpha_3,
	\end{equation*}
    where $\lambda_2 \circ \lambda_1$ and $\Lambda_2 \circ \Lambda_1$ denote the usual composite natural transformations between functors $\Ccal_1 \rightarrow \Ccal_2$ and $\Ecal_1 \rightarrow \Ecal_2$, respectively. 
\end{rem}

Here is the main result of this section:

\begin{prop}\label{prop_nat-lift}
	For $i = 1,2$ let $(\Dcal_i,\Acal_i;\beta_i)$ be an abelian representation; in addition, let $\alpha = (\phi,\Phi;\kappa)$ and $\alpha' = (\phi',\Phi';\kappa')$ be two exact $1$-morphisms of abelian representations $(\Dcal_1,\Acal_1;\beta_1) \rightarrow (\Dcal_2,\Acal_2;\beta_2)$. Suppose that we are given a $2$-morphism  $(\lambda,\Lambda): \alpha \rightarrow \alpha'$.
	Then there exists a unique natural transformation of functors $\A(\beta_1) \rightarrow \A(\beta_2)$
	\begin{equation*}
		\overline{\lambda}: \overline{\phi}(M) \rightarrow \overline{\phi'}(M)
	\end{equation*}
	such that the pair $(\lambda,\overline{\lambda})$ defines a $2$-morphism $\hat{\alpha} \rightarrow \hat{\alpha'}$ and the pair $(\overline{\lambda},\Lambda)$ defines a $2$-morphism $\tilde{\alpha} \rightarrow \tilde{\alpha'}$.
\end{prop}
\begin{proof}
	Since the abelian categories $\A(\beta_1)$ is generated by the image of $\Dcal_1$ under (co)kernels and the functors $\overline{\phi}, \overline{\phi'}: \A(\beta_1) \rightarrow \A(\beta_2)$ are exact, a natural transformation $\overline{\lambda}$ as in the statement is uniquely determined by the property that $(\lambda,\overline{\lambda})$ defines a $2$-morphism, provided it exists.	
	
	In order to show the existence of $\overline{\lambda}$, we let $\Ical$ denote the category with two objects $1$ and $2$ and exactly one non-identity arrow $1 \rightarrow 2$. Following the notation of \Cref{lem:Lcal}, consider the abelian representation $(\Dcal_2^{\Ical},\Acal_2^{\Ical};\beta_2^{\Ical})$ and its universal abelian factorization $\A(\beta^{\Ical})$.
	Applying \Cref{prop_lift-exfunct} to the exact morphism of abelian representations
	\begin{equation*}
		((\phi \xrightarrow{\lambda} \phi'), (\Phi \xrightarrow{\Lambda} \Phi'); (\kappa,\kappa')): (\Dcal_1,\Acal_1;\beta_1) \rightarrow (\Dcal_2^{\Ical},\Acal_2^{\Ical};\beta_2^{\Ical})
	\end{equation*} 
	obtained as in \Cref{rem:2mor-repr}, we get a canonical exact functor
	\begin{equation*}
		\overline{(\phi \xrightarrow{\lambda} \phi')}: \A(\beta_1) \rightarrow \A(\beta_2^{\Ical}) 
	\end{equation*}
	satisfying the equality of diagram representations $\Dcal_1 \rightarrow \A(\beta_2^{\Ical})$
	\begin{equation*}
		\pi_2^{(\Ical)} \circ (\phi \xrightarrow{\lambda} \phi') = \overline{(\phi \xrightarrow{\lambda} \phi')} \circ \pi_1,
	\end{equation*}
	together with a canonical natural isomorphism of functors $\A(\beta_1) \rightarrow \Acal^{\Ical}$
	\begin{equation*}
		\widetilde{(\kappa,\kappa')}: (\Phi \xrightarrow{\Lambda} \Phi') \circ \iota_1 \xrightarrow{\sim} \iota_2^{(\Ical)} \circ \overline{(\phi \xrightarrow{\lambda} \phi')}
	\end{equation*}
	which is uniquely characterized by the commutativity of the diagram of functors $\Dcal_1 \rightarrow \Acal_2^{\Ical}$
	\begin{equation*}
		\begin{tikzcd}
			(\Phi \xrightarrow{\Lambda} \Phi') \circ \beta_1 \arrow{rr}{(\kappa,\kappa')} \arrow[equal]{d} && \beta_2^{\Ical} \circ (\phi \xrightarrow{\lambda} \phi') \arrow[equal]{d} \\
			(\Phi \xrightarrow{\Lambda} \Phi') \circ \iota_1 \circ \pi_1 \arrow{r}{\widetilde{(\kappa,\kappa')}} & \iota_2^{(\Ical)} \circ \overline{(\phi \xrightarrow{\lambda} \phi')} \circ \iota_1 \arrow[equal]{r} & \iota_2^{(\Ical)} \circ \pi_2^{(\Ical)} \circ (\phi \xrightarrow{\lambda} \phi').
		\end{tikzcd}
	\end{equation*}
	Note that, as a consequence of \Cref{lem:Lcal}(2), the two composite exact functors
	\begin{equation*}
		\A(\beta_1) \xrightarrow{\overline{(\phi \xrightarrow{\lambda} \phi')}} \A(\beta_2^{\Ical}) \xrightarrow{\overline{1^*_{\Dcal_2}}} \A(\beta_2), \qquad \A(\beta_1) \xrightarrow{\overline{(\phi \xrightarrow{\lambda} \phi')}} \A(\beta_2^{\Ical}) \xrightarrow{\overline{2^*_{\Dcal_2}}} \A(\beta_2)
	\end{equation*}
	coincide with $\overline{\phi}$ and $\overline{\phi'}$, respectively. Hence we can define the sought-after natural transformation $\overline{\lambda}$ be declaring that its value on a given object $M_1 \in \A(\beta_1)$ is determined by the equation
	\begin{equation*}
		(\omega_{\Ical}^{\beta_2} \circ \overline{(\phi \xrightarrow{\lambda} \phi')})(M_1) = (\overline{\phi}(M_1),\overline{\phi'}(M_1);\overline{\lambda}(M_1)).
	\end{equation*}
	It is clear by construction that the diagram of functors $\Dcal_1 \rightarrow \A(\beta_2)$
	\begin{equation*}
		\begin{tikzcd}
			\overline{\phi} \circ \pi_1 \arrow{r}{\overline{\lambda}} \arrow[equal]{d} & \overline{\phi'} \circ \pi_1 \arrow[equal]{d} \\
			\pi_2 \circ \phi \arrow{r}{\lambda} & \pi_2 \circ \phi
		\end{tikzcd}
	\end{equation*}
	is commutative. Hence the pair $(\lambda,\overline{\lambda})$ defines a $2$-morphism of representations.
	Moreover, as a consequence of the uniqueness property of the natural transformation $\widetilde{(\kappa,\kappa')}$, we must have the equality
	\begin{equation*}
		\widetilde{(\kappa,\kappa')} = (\tilde{\kappa},\tilde{\kappa'}).
	\end{equation*} 
	By \Cref{rem:2mor-repr}, saying that the pair $(\tilde{\kappa},\tilde{\kappa'})$ defines a natural transformation of functors $\A(\beta_1) \rightarrow \Acal^{\Ical}$ is equivalent to saying that the pair $(\overline{\lambda},\Lambda)$ defines a $2$-morphism of representations.
\end{proof}

\begin{rem}\label{rem_comp-nat-lift}
	\begin{enumerate}
		\item Note that, while the existence of the natural transformation $\overline{\lambda}: \overline{\phi} \rightarrow \overline{\phi'}$ follows from the existence of $\Lambda$ (and from its compatibility with $\lambda$), the actual expression of $\overline{\lambda}$ only depends on $\lambda$.
		\item As a consequence of the uniqueness in the statement of \Cref{prop_nat-lift}, we see that the assignment $(\lambda,\Lambda) \mapsto \bar{\lambda}$ is compatible with composition of $2$-morphisms of representations (see \Cref{rem:comp-2mor}) in the following sense: Given, for $i = 1,2,3$, a morphism of abelian representations $\alpha_i = (\phi_i,\Phi_i;\theta_i): (\Dcal_1,\Acal_1;\beta_1) \rightarrow (\Dcal_2,\Acal_2;\beta_2)$ and given, for $i = 1,2$, a $2$-morphism $(\lambda_i,\Lambda_i): \alpha_i \rightarrow \alpha_{i+1}$, we have the equality of natural transformations of functors $\A(\beta_1) \rightarrow \Acal(\beta_2)$
		\begin{equation*}
			\overline{\lambda_2 \circ \lambda_1} = \overline{\lambda_2} \circ \overline{\lambda_1}.
		\end{equation*}
		\item As a consequence of the previous two points, we deduce that the natural transformation $\overline{\lambda}: \overline{\phi} \rightarrow \overline{\phi'}$ is invertible as soon as $\lambda: \phi \rightarrow \phi'$ is invertible.
	\end{enumerate}
\end{rem}

\subsection{The cases of adjunctions and equivalences}

The results collected so far can be applied to lifting adjunctions between exact functors; this method has been successfully used in the proof of \cite[\S~2.7]{IM19}, and the proof of our general result below is not harder than the latter. 

In order to state our result correctly, we need to use the notion of adjunction in its precise form. Thus we say that an \textit{adjunction}
\begin{equation*}
	F_1: \Ccal_1 \rightleftarrows \Ccal_2: F_2
\end{equation*}
consists of two functors $F_1: \Ccal_1 \rightarrow \Ccal_2$ and $F_2: \Ccal_2 \rightarrow \Ccal_1$ together with two natural transformations
\begin{equation*}
	\eta: \id_{\Ccal_2} \rightarrow F_1 \circ F_2, \quad \epsilon: F_2 \circ F_1 \rightarrow \id_{\Ccal_1}
\end{equation*}
called the \textit{unit} and the \textit{co-unit} of the adjunction, respectively, such that the triangular identities are satisfied.

\begin{defn}\label{defn_adj-repr}
	Let $R_1: \Ccal_1 \rightarrow \Ecal_1$ and $R_2: \Ccal_2 \rightarrow \Ecal_2$ be two functors. Suppose that we are given two adjunctions
	\begin{equation*}
		F_1: \Ccal_1 \rightleftarrows \Ccal_2: F_2, \quad G_1: \Ecal_1 \rightleftarrows \Ecal_2: G_2
	\end{equation*}
	fitting into morphism of representations
	\begin{equation*}
		\alpha_1 = (F_1,G_1;\kappa_1): (\Ccal_1,\Ecal_1;R_1) \rightarrow (\Ccal_2,\Ecal_2;R_2), \quad \alpha_2 = (F_2,G_2;\kappa_2^{-1}): (\Ccal_2,\Ecal_2;R_2) \rightarrow (\Ccal_1,\Ecal_1;R_1).
	\end{equation*}
	We say that the pair $(\alpha_1,\alpha_2)$ is \textit{adjoint} if the unit transformations
	\begin{equation*}
		\eta: \id_{\Ccal_1} \rightarrow F_2 \circ F_1, \qquad \eta: \id_{\Ecal_1} \rightarrow G_2 \circ G_1
	\end{equation*}
	define a $2$-morphism of representations $(\eta,\eta): (\id_{\Ccal_1},\id_{\Ecal_1};\id) \rightarrow (F_2 \circ F_1, G_2 \circ G_1, \kappa_2^{-1} \star \kappa_1)$ and the co-unit transformations
	\begin{equation*}
		\epsilon: F_1 \circ F_2 \rightarrow \id_{\Ccal_2}, \qquad \epsilon: G_1 \circ G_2 \rightarrow \id_{\Ecal_2}
	\end{equation*}
	define a $2$-morphism of representations $(\epsilon,\epsilon): (F_1 \circ F_2, G_1 \circ G_2; \kappa_1 \star \kappa_2^{-1}) \rightarrow (\id_{\Ccal_2},\id_{\Ecal_2};\id)$.
\end{defn}

\begin{rem}
	More explicitly, two morphisms $\alpha_1$ and $\alpha_2$ as in \Cref{defn_adj-repr} are adjoint if and only if the diagram of functors $\Ccal_1 \rightarrow \Ecal_1$
	\begin{equation*}
		\begin{tikzcd}
			R_1 \arrow{r}{\eta} \arrow{d}{\eta} & R_1 \circ F_2 \circ F_1 \arrow{d}{\kappa_2} \\
			G_2 \circ G_1 \circ R_1 \arrow{r}{\kappa_1} & G_2 \circ R_2 \circ F_1
		\end{tikzcd}
	\end{equation*}
	and the diagram of functors $\Ccal_2 \rightarrow \Ecal_2$
	\begin{equation*}
		\begin{tikzcd}
			G_1 \circ R_1 \circ F_2 \arrow{d}{\kappa_2} \arrow{r}{\kappa_1} & R_2 \circ F_1 \circ F_2  \arrow{d}{\epsilon} \\
			G_1 \circ G_2 \circ R_2 \arrow{r}{\epsilon} & R_2
		\end{tikzcd}
	\end{equation*}
	are commutative. One can check that requiring the latter condition is equivalent to requiring that the natural isomorphism $\kappa_2$ equal the composite
	\begin{equation*}
		R_1 F_2 \xrightarrow{\eta} G_2 G_1 \circ R_1 F_2 \xrightarrow{\kappa_1} G_2 R_2 \circ F_1 F_2 \xrightarrow{\epsilon} G_2 R_2
	\end{equation*}
	or, equivalently, that the natural isomorphism $\kappa_1$ equal the composite
	\begin{equation*}
		G_1 R_1 \xrightarrow{\eta} G_1 R_1 \circ F_2 F_1 \xrightarrow{\kappa_2} G_1 G_2 \circ R_2 F_1 \xrightarrow{\epsilon} R_2 F_1.
	\end{equation*}
\end{rem}

The following result asserts that compatible adjunctions between abelian representations canonically induce adjunctions between the corresponding universal abelian factorizations:

\begin{cor}\label{cor_lift-adj}
	For $i = 1,2$ let $(\Dcal_i,\Acal_i;\beta_i)$ be an abelian representation.
	Suppose that we are given two exact morphisms of abelian representations
	\begin{equation*}
		\alpha_1 = (\phi_1,\Phi_1;\kappa_1): (\Dcal_1,\Acal_1;\beta_1) \rightarrow (\Dcal_2,\Acal_2;\beta_2), \quad \alpha_2 = (\phi_2,\Phi_2;\kappa_2^{-1}): (\Dcal_2,\Acal_2;\beta_2) \rightarrow (\Dcal_1,\Acal_1;\beta_1)
	\end{equation*}
	such that the pair $(\alpha_1,\alpha_2)$ is adjoint in the sense of \Cref{defn_adj-repr}. Then the exact functors $\bar{\phi}_1$ and $\bar{\phi}_2$ fit in a unique way into an adjunction
	\begin{equation}\label{adj-tildePhi}
		\overline{\phi_1}: \A(\beta_1) \rightleftarrows \A(\beta_2): \overline{\phi_2}
	\end{equation}
	such that the pairs $(\hat{\alpha}_1,\hat{\alpha}_2)$ and $(\tilde{\alpha}_1,\tilde{\alpha}_2)$ are adjoint as well.
\end{cor}
\begin{proof}
	Applying \Cref{prop_nat-lift}, and taking into account \Cref{rem_comp-lift}(2), we obtain a natural transformation of functors $\A(\beta_1) \rightarrow \A(\beta_1)$
	\begin{equation*}
		\overline{\eta}: \id_{\A(\beta_1)} \rightarrow \overline{\phi_2 \circ \phi_1} = \overline{\phi_2} \circ \overline{\phi_1}
	\end{equation*}
	defining $2$-morphisms of abelian representations
	\begin{equation*}
		(\eta,\overline{\eta}): (\id_{\Dcal_1},\id_{\A(\Dcal_1)}; \id) \rightarrow (\phi_2 \circ \phi_1,\overline{\phi_2} \circ \overline{\phi_1}; \id), \qquad (\overline{\eta},\eta): (\id_{\Acal(\Dcal_1)},\id_{\Acal_1};\tilde{\kappa}_2 \star \tilde{\kappa}_1).
	\end{equation*}
	In the same way, we obtain a natural transformation of functors $\A(\beta_2) \rightarrow \A(\beta_2)$
	\begin{equation*}
		\overline{\epsilon}: \overline{\phi_1} \circ \overline{\phi_2} = \overline{\phi_1 \circ \phi_2} \rightarrow \id_{\A(\beta_2)}
	\end{equation*}
	defining $2$-morphisms of abelian representations
	\begin{equation*}
		(\epsilon,\overline{\epsilon}): (\phi_1 \circ \phi_2,\overline{\phi_1} \circ \overline{\phi_2}; \id), \qquad (\overline{\epsilon},\epsilon): (\overline{\phi_1} \circ \overline{\phi_2}, \Phi_1 \circ \Phi_2; \tilde{\kappa}_1 \star \tilde{\kappa}_2^{-1}).
	\end{equation*}
	As a consequence of the uniqueness part of \Cref{prop_nat-lift}, and taking into account \Cref{rem_comp-nat-lift}(2), we see that the composite natural transformation of functors $\A(\beta_1) \rightarrow \A(\beta_2)$
	\begin{equation*}
		\overline{\phi_1} \xrightarrow{\overline{\eta}} \overline{\phi_1} \circ \overline{\phi_2} \circ \overline{\phi_1} \xrightarrow{\overline{\epsilon}} \overline{\phi_1}
	\end{equation*}
	as well as the composite natural transformation of functors $\A(\beta_2) \rightarrow \A(\beta_1)$
	\begin{equation*}
		\overline{\phi_2} \xrightarrow{\overline{\eta}} \overline{\phi_2} \circ \overline{\phi_1} \circ \overline{\phi_2} \xrightarrow{\overline{\epsilon}} \overline{\phi_2}
	\end{equation*}
	are the identity natural transformations. Thus $\overline{\eta}$ and $\overline{\epsilon}$ satisfy the triangular identities, and therefore define the adjunction \eqref{adj-tildePhi}.
\end{proof}

In the same spirit, the previous results can be exploited to deduce equivalences between universal abelian factorizations. By definition, giving an equivalence between two categories $\Ccal_1$ and $\Ccal_2$ means giving two functors
\begin{equation*}
	F_1: \Ccal_1 \rightarrow \Ccal_2, \qquad F_2: \Ccal_2 \rightarrow \Ccal_1
\end{equation*}
and two natural isomorphisms
\begin{equation*}
	\eta: \id_{\Ccal_1} \xrightarrow{\sim} F_2 \circ F_1, \qquad \epsilon: F_1 \circ F_2 \xrightarrow{\sim} \id_{\Ccal_2}.
\end{equation*}
Using the Yoneda Lemma, one can see that this is exactly the same as giving an adjunction
\begin{equation*}
	F_1: \Ccal_1 \leftrightarrows \Ccal_2: F_2
\end{equation*}
where both $F_1$ and $F_2$ are fully faithful: indeed, the fully faithfulness of $F_1$ is equivalent to the invertibility of the unit, while the fully faithfulness of $F_2$ is equivalent to the invertibility of the co-unit. In the following result, it is convenient to view equivalences in this way.

\begin{cor}\label{cor_lift-equiv}
	Keep the notation and assumptions of \Cref{cor_lift-adj}, and suppose in addition that the adjunction
	\begin{equation*}
		\phi_1: \Dcal_1 \leftrightarrows \Dcal_2: \phi_2
	\end{equation*}
	is an equivalence. Then the adjunction
	\begin{equation*}
		\overline{\phi_1}: \A(\beta_1) \leftrightarrows \A(\beta_2): \overline{\phi_2}
	\end{equation*}
	is an equivalence as well.
\end{cor}
\begin{proof}
	We have to show that the two natural transformations
	\begin{equation*}
		\overline{\eta}: \id_{\A(\beta_1)} \rightarrow \overline{\phi_2 \circ \phi_1} = \overline{\phi_2} \circ \overline{\phi_1}, \qquad \overline{\epsilon}: \overline{\phi_1} \circ \overline{\phi_2} = \overline{\phi_1 \circ \phi_2} \rightarrow \id_{\A(\beta_2)}
	\end{equation*}
	constructed in the proof of \Cref{cor_lift-adj} are invertible. Taking into account \Cref{rem_comp-nat-lift}(3), this follows from the hypothesis that the two natural transformations
	\begin{equation*}
		\eta: \id_{\Dcal_1} \rightarrow \phi_2 \circ \phi_1, \qquad \epsilon: \phi_1 \circ \phi_2 \rightarrow \id_{\Dcal_2}
	\end{equation*}
	are invertible.
\end{proof}

\subsection{Compatibility with diagram categories}\label{subsec:nat-dia}

We conclude by discussing the compatibility of \Cref{prop_nat-lift} with families of abelian representations indexed by some small category $\Lcal$; as a particular application, we show that natural exact sequences lift to natural exact sequences.

For notational convenience, we use the following definition:

\begin{defn}
	Let $\Lcal$ be a (small) category.
	\begin{enumerate}
		\item We call \textit{$\Lcal$-family of $1$-morphisms of representations} from $(\Ccal_1,\Ecal_1;R_1)$ to $(\Ccal_2,\Ecal_2;R_2)$ the datum of
		\begin{itemize}
			\item for every $l \in \Lcal$, a morphism of representations $\alpha_l = (F_l,G_l;\kappa_l): (\Ccal_1,\Ecal_1;R_1) \rightarrow (\Ccal_2,\Ecal_2;R_2)$
			\item for every arrow $\gamma: l \rightarrow l'$ in $\Lcal$, a $2$-morphism $(\lambda_{\gamma},\Lambda_{\gamma}): \alpha_{l} \rightarrow \alpha_{l'}$
		\end{itemize}
		such that the assignments $\gamma \mapsto \lambda_{\gamma}$ and $\gamma \mapsto \Lambda_{\gamma}$ are compatible with identities and composition in the obvious sense. 
		We write it as $\alpha: \left\{\alpha_l: (F_l,G_l;\kappa_l)\right\}_{l \in \Lcal}$, leaving the natural transformations $\lambda_{\gamma}$ and $\Lambda_{\gamma}$ implicit.
		\item In case the representations $(\Ccal_1,\Ecal_1;R_1)$ and $(\Ccal_2,\Ecal_2;R_2)$ are abelian, we say that an $\Lcal$-family of morphisms of representations $\alpha: (\Ccal_1,\Ecal_1;R_1) \rightarrow (\Ccal_2,\Ecal_2;R_2)$ is \textit{exact} of all the individual $\alpha_l: \Ecal_1 \rightarrow \Ecal_2$ for $l \in \Lcal$ are exact in the sense of \Cref{defn:1-mor_repr}(1)(b).
	\end{enumerate}
\end{defn}

The lifting result of \Cref{prop_lift-exfunct} is compatible with families of $1$-morphisms in the following sense:

\begin{lem}\label{cor:lift-Cfam}
	Let $(\Dcal_1,\Acal_1;\beta_1)$ and $(\Dcal_2,\Acal_2;\beta_2)$ be two abelian representations. Let $\Lcal$ be a (small) category, and suppose that we are given an $\Lcal$-family of exact morphisms of abelian representations
	\begin{equation*}
		\left\{\alpha_l = (\phi_l,\Phi_l;\kappa_l): (\Dcal_1,\Acal_1;\beta_1) \rightarrow (\Dcal_2,\Acal_2;\beta_2) \right\}_{l \in \Lcal}.
	\end{equation*}  
	Then the two collections of exact morphisms of abelian representations
	\begin{equation*}
		\left\{\hat{\alpha}_l = (\phi_l,\overline{\phi_l};\id): (\Dcal_1,\A(\beta_1);\pi_1) \rightarrow (\Dcal_2,\A(\beta_2);\pi_2) \right\}_{l \in \Lcal}
	\end{equation*}  
	and
	\begin{equation*}
		\left\{\tilde{\alpha}_l = (\overline{\phi_l},\Phi_l;\tilde{\kappa_l}): (\A(\beta_1),\Acal_1;\iota_1) \rightarrow (\A(\beta_2),\Acal_2;\iota_2) \right\}_{l \in \Lcal}
	\end{equation*}  
	canonically assemble into $\Lcal$-families of exact morphisms of abelian representations.
\end{lem}
\begin{proof}
	For every fixed object $C \in \Ccal$, we get by restriction a morphism of abelian representations $\alpha_C := (\phi_C,\Phi_C;\kappa_C): (\Dcal_1,\Acal_1;\beta_1) \rightarrow (\Dcal_2,\Acal_2;\beta_2)$.
	Applying \Cref{prop_lift-exfunct}, we obtain an exact functor
	\begin{equation*}
		\overline{\phi_C}: \A(\beta_1) \rightarrow \A(\beta_2),
	\end{equation*}  
	uniquely determined by the equality of additive functors $\Dcal_1 \rightarrow \A(\beta_2)$
	\begin{equation*}
		\pi_2 \circ \phi_C = \overline{\phi_C} \circ \pi_1,
	\end{equation*} 
	as well as a natural isomorphism of functors $\A(\beta_1) \rightarrow \Acal_2$
	\begin{equation*}
		\tilde{\kappa}_C: \Phi_C \circ \iota_1 \xrightarrow{\sim} \iota_2 \circ \overline{\phi}_C,
	\end{equation*}
	uniquely determined by the commutativity of the diagram of functors $\Dcal_1 \rightarrow \Acal$
	\begin{equation*}
		\begin{tikzcd}
			\Phi_C \circ \beta_1 \arrow[equal]{d} \arrow{rr}{\kappa_C} && \beta_2 \circ \phi_C \arrow[equal]{d} \\
			\Phi_C \circ \iota_1 \circ \pi_1 \arrow{r}{\tilde{\kappa}_C} & \iota \circ \overline{\phi_C} \circ \pi_1 \arrow[equal]{r} & \iota_2 \circ \pi_2 \circ \phi_C.
		\end{tikzcd}
	\end{equation*}
	Moreover, for every morphism $\gamma: C \rightarrow C'$ in $\Ccal$, the natural transformation of additive functors $\Dcal_1 \rightarrow \Dcal_2$
	\begin{equation*}
		\lambda_{\gamma}: \phi_C \rightarrow \phi_{C'}
	\end{equation*} 
	and the natural transformation of exact functors $\Acal_1 \rightarrow \Acal_2$
	\begin{equation*}
		\Lambda_{\gamma}: \Phi_C \rightarrow \Phi_{C'}
	\end{equation*}
	define a $2$-morphism $(\lambda_{\gamma},\Lambda_{\gamma}): \alpha_C \rightarrow \alpha_{C'}$. Applying \Cref{prop_nat-lift}, we obtain a natural transformation of exact functors $\A(\beta_1) \rightarrow \A(\beta_2)$
	\begin{equation*}
		\overline{\lambda_{\gamma}}: \overline{\phi_C} \rightarrow \overline{\phi_{C'}},
	\end{equation*}
	uniquely determined by the property that the pair $(\lambda_{\gamma},\overline{\lambda_{\gamma}})$ defines a $2$-morphism $\hat{\alpha}_C \rightarrow \hat{\alpha}_{C'}$ and the pair $(\overline{\lambda_{\gamma}},\Lambda_{\gamma})$ defines a $2$-morphism $\tilde{\alpha}_C \rightarrow \tilde{\alpha}_{C'}$. 
\end{proof}

As a particular case, one can consider $\Z$-indexed families of $1$-morphisms, for instance complexes of $1$-morphisms. We check that exact complexes of $1$-morphisms of representations give rise to exact complexes between universal abelian factorizations:

\begin{cor}\label{lem:lift-exseq}
	Let $(\Dcal_1,\Acal_1;\beta_1)$ and $(\Dcal_2,\Acal_2;\beta_2)$ be two abelian representations. Suppose that we are given, for every $n \in \Z$, an exact morphism of abelian representations
	\begin{equation*}
		\alpha_n = (\phi_n,\Phi_n;\kappa_n): (\Dcal_1,\Acal_1;\beta_1) \rightarrow (\Dcal_2,\Acal_2;\beta_2).
	\end{equation*} 
	Suppose in addition that we are given, for every $n \in \Z$, a $2$-morphism $(\lambda_n,\Lambda_n): \alpha_n \rightarrow \alpha_{n+1}$, such that the sequence of functors $\Acal_1 \rightarrow \Acal_2$
	\begin{equation*}
		\dots \xrightarrow{\Lambda_{n-2}} \Phi_{n-1} \xrightarrow{\Lambda_{n-1}} \Phi_n \xrightarrow{\Lambda_n} \Phi_{n+1} \xrightarrow{\Lambda_{n+1}} \dots
	\end{equation*}
	is exact. Then the sequence of functors $\A(\beta_1) \rightarrow \A(\beta_2)$
	\begin{equation*}
		\dots \xrightarrow{\overline{\lambda_{n-2}}} \overline{\phi_{n-1}} \xrightarrow{\overline{\lambda_{n-1}}} \overline{\phi_n} \xrightarrow{\overline{\lambda_n}} \overline{\phi_{n+1}} \xrightarrow{\overline{\lambda_{n+1}}} \dots
	\end{equation*}
	is exact as well.
\end{cor}
\begin{proof}
	Since the functor $\iota_2: \A(\beta_2) \hookrightarrow \Acal_2$ is exact and faithful, it suffices to show that the sequence of functors $\A(\beta_1) \rightarrow \Acal_2$
	\begin{equation*}
		\dots \xrightarrow{\overline{\lambda_{n-2}}} \iota_2 \circ \overline{\phi_{n-1}} \xrightarrow{\overline{\lambda_{n-1}}} \iota_2 \circ \overline{\phi_n} \xrightarrow{\overline{\lambda_n}} \iota_2 \circ \overline{\phi_{n+1}} \xrightarrow{\overline{\lambda_{n+1}}} \dots
	\end{equation*}
	is exact. But the latter is naturally isomorphic to the sequence of functors
	\begin{equation*}
		\dots \xrightarrow{\Lambda_{n-2}} \Phi_{n-1} \circ \iota_1 \xrightarrow{\Lambda_{n-1}} \Phi_n \circ \iota_1 \xrightarrow{\Lambda_n} \Phi_{n+1} \circ \iota_1 \xrightarrow{\Lambda_{n+1}} \dots,
	\end{equation*}
	which is exact as an immediate consequence of the hypothesis.  
\end{proof}

\section{Extension to the multi-linear setting}\label{sect:multi}

In this section, we explain how the universal property of universal abelian factorizations can be used to define multi-exact functors and natural transformations of such. The results below generalize those of \Cref{sect:lift-exfun} and \Cref{sect:lift-nat}. In order to prove them, we need to combine the main results of the previous sections, mostly via \Cref{cor:lift-Cfam}. 

The problem  of lifting multi-exact functors to the level of universal abelian representations in a systematic way was addressed in \cite{BVHP20}. The starting point is \cite[Prop.~1.4]{BVHP20}, which gives a canonical method to extend multi-linear functors from additive categories to their abelian hulls and shows that the resulting extension is right-exact in each variable. As pointed out in \cite[Ex.~1.12]{BVHP20}, the extension is not left-exact in each variable in general; however, as explained in \cite[Prop.~1.13]{BVHP20}, one can still use this method to construct multi-linear functors between universal abelian factorizations associated to certain representations, at least under suitable hypotheses.

Here we describe a slightly different construction that is tailored to our applications. The main difference with \cite{BVHP20} is that we bypass the level of abelian hulls and get directly a multi-exact functor on universal abelian factorizations. Note that our construction works only if the given multi-linear functor on the underlying abelian categories is multi-exact, while the result of \cite[Prop.~1.13]{BVHP20} works under weaker assumptions (see \cite[Defn.~1.9]{BVHP20}) and can be applied, for instance, to Nori motives with integral coefficients (see \cite[Thm.~2.20, Cor.~4.4]{BVHP20}). In all those cases where both methods apply, the multi-exact functors obtained via the two distinct approaches necessarily coincide, since both of them satisfy the same uniqueness property.

\subsection{Lifting multi-exact functors}

Given three categories $\Ccal, \Ccal_2$ and $\Ccal$, a functor $F: \Ccal_1 \times \Ccal_2 \rightarrow \Ccal$ can be regarded in two equivalent ways:
\begin{itemize}
	\item[–] as a $\Ccal_1$-family of functors $\left\{F(C_1,-): \Ccal_2 \rightarrow \Ccal\right\}_{C_1 \in \Ccal_1}$;
	\item[–] as a $\Ccal_2$-family of functors $\left\{F(-,C_2): \Ccal_1 \rightarrow \Ccal\right\}_{C_2 \in \Ccal_2}$.
\end{itemize}
This observation generalizes to functors from arbitrary finite products of categories. In the course of this section we systematically adopt this point of view, and we often switch the roles of the two variables.

\begin{defn}
	\begin{enumerate}
		\item Let $\Ccal_1, \dots, \Ccal_n$ and $\Ccal$ be additive categories. A functor
		\begin{equation*}
			F: \Ccal_1 \times \dots \times \Ccal_n \rightarrow \Ccal
		\end{equation*}
		is \textit{multi-additive} if it is additive separately with respect to each variable.
		\item Let $\Acal_1, \dots, \Acal_n$ and $\Acal$ be abelian categories. A multi-additive functor
		\begin{equation*}
			F: \Acal_1 \times \dots \times \Acal_n \rightarrow \Acal
		\end{equation*}
		is \textit{multi-exact} if it is exact separately with respect to each variable.
	\end{enumerate}
\end{defn}

\begin{defn}
	For each $i = 1,\dots,n$ let $(\Ccal_i,\Ecal_i;R_i)$ be an additive representation; let $(\Ccal,\Ecal;R)$ be a further additive representation. 
	\begin{enumerate}
		\item We call \textit{multi-linear $1$-morphism of representations} from $(\Ccal_1,\Ecal_1;R_1) \times \dots \times (\Ccal_n,\Ecal_n;R_n)$ to $(\Ccal,\Ecal;R)$ a triple $\alpha = (F,G;\kappa)$ consisting of two multi-additive functors
		\begin{equation*}
			F: \Ccal_1 \times \dots \times \Ccal_n \rightarrow \Ccal, \qquad G: \Ecal_1 \times \dots \times \Ecal_n \rightarrow \Ecal
		\end{equation*}
		and a natural isomorphism of functors $\Ccal_1 \times \dots \times \Ccal_n \rightarrow \Ecal$
		\begin{equation*}
			\kappa: G \circ (R_1,\dots,R_n) \xrightarrow{\sim} R \circ F.
		\end{equation*}
		\item In case the representations $(\Ccal_i,\Ecal_i;R_i)$, $i = 1,\dots,n$ and $(\Ccal,\Ecal;R)$ are all abelian, we say that a $1$-morphism
		\begin{equation*}
			\alpha = (F,G;\kappa): (\Ccal_1,\Ecal_1;R_1) \times \dots \times (\Ccal_n,\Ecal_n;R_n) \rightarrow (\Ccal,\Ecal;R)
		\end{equation*}  
		is \textit{multi-exact} if the multi-additive functor $G: \Ecal_1 \times \dots \times \Ecal_n \rightarrow \Ecal$ is multi-exact.
	\end{enumerate}
\end{defn}

Here is the multi-linear extension of \Cref{prop_lift-exfunct}:

\begin{prop}\label{prop:lift-multiex}
	For each $i = 1,\dots,n$ let $(\Dcal_i,\Acal_i;\beta_i)$ be an abelian representation, and let $(\Dcal,\Acal;\beta)$ be a further abelian representation. Suppose that we are given a multi-exact $1$-morphism of abelian representations $\alpha = (\phi,\Phi;\kappa): (\Dcal_1,\Acal_1;\beta_1) \times \dots \times (\Dcal_n,\Acal_n;\beta_n) \rightarrow (\Dcal,\Acal;\beta)$. Then:
	\begin{enumerate}
		\item There exists a unique multi-exact functor $\overline{\phi}: \A(\beta_1) \times \dots \times \A(\beta_n) \rightarrow \A(\beta)$ satisfying the equality of multi-additive functors $\Dcal_1 \times \dots \times \Dcal_n \rightarrow \A(\beta)$
		\begin{equation}\label{equal:pi-phi=barphi-pi1n}
			\pi \circ \phi = \overline{\phi} \circ (\pi_1,\dots,\pi_n).
		\end{equation}
		\item There exists a unique natural isomorphism of functors $\A(\beta_1) \times \dots \A(\beta_n) \rightarrow \Acal$
		\begin{equation*}
			\tilde{\kappa}: \Phi \circ (\iota_1,\dots,\iota_n) \xrightarrow{\sim} \iota \circ \overline{\phi}
		\end{equation*}
		making the diagram of functors $\Dcal_1 \times \dots \times \Dcal_n \rightarrow \Acal$
		\begin{equation}\label{dia-D1nA1n}
			\begin{tikzcd}
				\Phi \circ (\beta_1,\dots,\beta_n) \arrow[equal]{d} \arrow{rr}{\kappa} && \beta \circ \phi \arrow[equal]{d} \\
				\Phi \circ (\iota_1,\dots,\iota_n) \circ (\pi_1,\dots,\pi_n) \arrow[dashed]{r}{\tilde{\kappa}} & \iota \circ \overline{\phi} \circ (\pi_1,\dots,\pi_n) \arrow[equal]{r} & \iota \circ \pi \circ \phi
			\end{tikzcd}
		\end{equation}
		commute.
	\end{enumerate}
\end{prop}
\begin{proof}
	Since, for each $i = 1,\dots,n$, the abelian category $\A(\beta_i)$ is generated under (co)kernels by the image of $\Dcal_i$, a multi-exact functor $\bar{\phi}$ as in (1) is uniquely determined by the equality \eqref{equal:pi-phi=barphi-pi1n}, provided it exists. For the same reason, a natural isomorphism $\tilde{\kappa}$ as in (2) is uniquely determined by the commutativity of the diagram \eqref{dia-D1nA1n}, provided it exists.
	For simplicity, we only show the existence of $\overline{\phi}$ and $\tilde{\kappa}$ in the case $n = 2$; the general case is just a notationally more intricate variant. 
	
	In the case $n=2$, regard the bi-additive functor $\phi: \Dcal_1 \times \Dcal_2 \rightarrow \Dcal$ as a $\Dcal_2$-family of additive functors
	\begin{equation*}
		\left\{\phi_{D_2}: \Dcal_1 \rightarrow \Dcal, \; D_1 \rightsquigarrow \phi(D_1,D_2)\right\}_{D_2 \in \Dcal_2}.
	\end{equation*} 
	Similarly, regard the bi-exact functor $\Phi: \Acal_1 \times \Acal_2 \rightarrow \Acal$ as an $\Acal_2$-family of exact functors
	\begin{equation*}
		\left\{\Phi_{A_2}: \Acal_1 \rightarrow \Acal, \; A_1 \rightsquigarrow \Phi(A_1,A_2)\right\}_{A_2 \in \Acal_2},
	\end{equation*} 
	and, by restriction, as a $\Dcal_2$-family of exact functors
	\begin{equation*}
		\left\{\Phi_{\beta_2(D_2)}: \Acal_1 \rightarrow \Acal, \; A_1 \rightsquigarrow \Phi(A_1,\beta_2(D_2))\right\}_{D_2 \in \Dcal_2}.
	\end{equation*} 
	Together, they give rise to a $\Dcal_2$-family of exact $1$-morphisms of abelian representations 
	\begin{equation*}
		\left\{\alpha_{D_2} = (\phi_{D_2},\Phi_{\beta_2(D_2)};\kappa_{D_2}): (\Dcal_1,\Acal_1\beta_1) \rightarrow (\Dcal,\Acal;\beta) \right\}_{D_2 \in \Dcal_2}.
	\end{equation*}
	Applying \Cref{cor:lift-Cfam}, we obtain a $\Dcal_2$-family of exact functors
	\begin{equation}\label{D_2-fam:barphi}
		\left\{\bar{\phi}_{D_2}: \A(\beta_1) \rightarrow \A(\beta)\right\}_{D_2 \in \Dcal_2}
	\end{equation}
	giving rise to a $\Dcal_2$-family of exact morphisms of abelian representations 
	\begin{equation}\label{D_2-fam:tildealpha}
		\left\{\tilde{\alpha}_{D_2} = (\bar{\phi}_{D_2},\Phi_{\beta_2(D_2)}, \tilde{\kappa}_{D_2}): (\A(\beta_1),\Acal_1;\iota_1) \rightarrow (\A(\beta),\Acal;\iota) \right\}_{D_2 \in \Dcal_2}.
	\end{equation}
	We now switch the roles of the two variables, regarding \eqref{D_2-fam:barphi} as an $\A(\beta_1)$-family of additive functors 
	\begin{equation*}
		\left\{\bar{\phi}_{M_1}: \Dcal_2 \rightarrow \A(\beta), \; D_2 \rightsquigarrow \bar{\phi}_{D_2}(M_1) \right\}_{M_1 \in \A(\beta_1)}
	\end{equation*}
	and \eqref{D_2-fam:tildealpha} as an $\A(\beta_1)$-family of morphisms of abelian representations
	\begin{equation*}
		\left\{\tilde{\alpha}_{M_1}: (\bar{\phi}_{M_1},\Phi_{\iota_1(M_1)};\tilde{\kappa}_{M_1}): (\Dcal_2,\Acal_2;\beta_2) \rightarrow (\A(\beta),\Acal;\iota)\right\}_{M_1 \in \A(\beta_1)}.
	\end{equation*}
	Applying \Cref{cor:lift-Cfam} again, we obtain an $\A(\beta_1)$-family of exact functors
	\begin{equation}\label{M_1-fam:barphi}
		\left\{\bar{\phi}_{M_1}: \A(\beta_2) \rightarrow \A(\beta)\right\}_{M_1 \in \A(\beta_1)}
	\end{equation}
	giving rise to an $\A(\beta_1)$-family of exact morphisms of abelian representations (written in the same way, for notational convenience)
	\begin{equation}\label{M_1-fam:tildealpha}
		\left\{\tilde{\alpha}_{M_1}: (\bar{\phi}_{M_1},\Phi_{\iota_1(M_1)};\tilde{\kappa}_{M_1}): (\A(\beta_2),\Acal_2;\iota_2) \rightarrow (\A(\beta),\Acal;\iota)\right\}_{M_1 \in \A(\beta_1)}.
	\end{equation}
	We can finally regard \eqref{M_1-fam:barphi} as a bi-additive functor
	\begin{equation*}
		\overline{\phi}: \A(\beta_1) \times \A(\beta_2) \rightarrow \A(\beta)
	\end{equation*}  
	which is exact with respect to the second variable and, by construction, satisfies the equality \eqref{equal:pi-phi=barphi-pi1n}. The natural isomorphisms $\tilde{\kappa}_{M_1}$ in \eqref{M_1-fam:tildealpha} assemble into a natural isomorphism of functors $\A(\beta_1) \times \A(\beta_2) \rightarrow \Acal$
	\begin{equation*}
		\tilde{\kappa}: \Phi \circ (\iota_1,\iota_2) \xrightarrow{\sim} \iota \circ \overline{\phi}
	\end{equation*}
	which, by construction, makes the diagram \eqref{dia-D1nA1n} commutative (with respect to the bi-additive functor $\overline{\phi}$ just obtained). Since the underlying bi-additive functor $\Phi: \Acal_1 \times \Acal_2 \rightarrow \Acal$ is bi-exact, using the commutativity of the diagram \eqref{dia-D1nA1n} together with the fact that $\iota: \A(\beta) \rightarrow \Acal$ is faithful and exact we deduce that the functor $\overline{\phi}$ is bi-exact as well. This concludes the proof.
\end{proof}

\begin{rem}
	The analogue of \Cref{rem_comp-lift} applies, namely: 
	\begin{enumerate}
		\item Note that, while the existence of $\overline{\phi}$ follows from the existence of the exact functor $\Phi$ (and of the natural isomorphism $\kappa$), the actual expression of $\overline{\phi}$ only depends on $\phi$. 
		\item As a consequence of the uniqueness in the two statements of \Cref{prop_lift-exfunct}, we see that the assignments $\alpha \mapsto \hat{\alpha}$ and $\alpha \mapsto \tilde{\alpha}$ are both compatible with composition of multi-exact $1$-morphisms of representations in the expected way.
	\end{enumerate}
\end{rem}

It is convenient to extend the flexible variant of \Cref{prop_lift-exfunct} given by \Cref{cor:lifting-flexible} to the case of multi-linear representations.

\begin{cor}
	For each $i = 1,\dots,n$ let $(\Dcal_i,\Acal_i;\beta_i)$ be an abelian representation; let $(\Dcal,\Acal;\beta)$ be a further abelian representation. Suppose that we are given a multi-exact $1$-morphism of abelian representations
	\begin{equation*}
		\alpha:= (\phi,\Phi;\kappa): (\Dcal_1,\Acal_1;\beta_1) \times \dots \times (\Dcal_n,\Acal_n;\beta_n) \rightarrow (\Dcal,\Acal;\beta).
	\end{equation*}
	In addition, suppose that we are given
	\begin{itemize}
		\item a multi-exact functor $\psi: \A(\beta_1) \times \dots \times \A(\beta_n) \rightarrow \A(\beta)$,
		\item a natural isomorphism of functors $\Dcal_1 \times \dots \times \Dcal_n \rightarrow \A(\beta)$
		\begin{equation*}
			\chi: \psi \circ (\pi_1,\dots,\pi_n) \xrightarrow{\sim} \pi \circ \phi,
		\end{equation*}
		\item a natural isomorphism of functors $\A(\beta_1) \times \dots \times \A(\beta_n) \rightarrow \Acal$
		\begin{equation*}
			\xi: \Phi \circ (\iota_1,\dots,\iota_n) \xrightarrow{\sim} \iota \circ \psi
		\end{equation*}
	\end{itemize}
	satisfying the equality of $1$-morphisms between multi-linear representations 
	\begin{equation*}
		(\iota_1 \times \dots \times \iota_n,\iota;\xi) \circ (\pi_1 \times \dots \times \pi_n,\pi;\xi) = (\beta_1 \times \dots \times \beta_n,\beta;\kappa).
	\end{equation*}
	Then there exists a unique natural isomorphism of functors $\A(\beta_1) \times \dots \A(\beta_n) \rightarrow \A(\beta)$
	\begin{equation*}
		\tilde{\chi}: \overline{\phi} \xrightarrow{\sim} \psi
	\end{equation*}
	making the diagram of multi-linear representations $\A(\beta_1) \times \dots \times \A(\beta_n) \rightarrow \Acal$
	\begin{equation*}
		\begin{tikzcd}
			\iota \circ \psi \arrow{dr}{\xi} \arrow[dashed]{rr}{\tilde{\chi}} && \iota \circ \overline{\phi} \arrow{dl}{\tilde{\kappa}} \\
			& \Phi \circ (\iota_1 \times \dots \times \iota_n)
		\end{tikzcd}
	\end{equation*}
	commute.
\end{cor}
\begin{proof}
	This follows from \Cref{prop:lift-multiex} in the same way as \Cref{cor:lifting-flexible} follows from \Cref{prop_lift-exfunct}. We leave the details to the interested reader.
\end{proof}

\subsection{Lifting multi-linear natural transformations}

Here is the multi-linear extension of \Cref{prop_nat-lift}:

\begin{prop}\label{prop:lift-multinat}
	For each $i = 1,\dots,n$ let $(\Dcal_i,\Acal_i;\beta_i)$ be an abelian representation, and let $(\Dcal,\Acal;\beta)$ be a further abelian representation; in addition, let $\alpha = (\phi,\Phi;\kappa)$ and $\alpha' = (\phi',\Phi';\kappa')$ be two multi-exact $1$-morphisms of abelian representations $(\Dcal_1,\Acal_1;\beta_1) \times \dots \times (\Dcal_n,\Acal_n;\beta_n) \rightarrow (\Dcal,\Acal;\beta)$. 
	Suppose that we are given a $2$-morphism $(\lambda,\Lambda): \alpha \rightarrow \alpha'$. Then there exists a unique natural transformation of functors $\A(\beta_1) \times \dots \times \A(\beta_n) \rightarrow \A(\beta)$
	\begin{equation*}
		\overline{\lambda}: \overline{\phi}(M_1,\dots,M_n) \rightarrow \overline{\phi'}(M_1,\dots,M_n)
	\end{equation*}
	such that the pair $(\lambda,\overline{\lambda})$ defines a $2$-morphism $\hat{\alpha} \rightarrow \hat{\alpha'}$ and the pair $(\overline{\lambda},\Lambda)$ defines a $2$-morphism $\tilde{\alpha} \rightarrow \tilde{\alpha'}$.
\end{prop}
\begin{proof}
	Since, for each $i = 1,\dots,n$, the abelian category $\A(\beta_i)$ is generated under (co)kernels by the image of $\Dcal_i$, and since the functors $\overline{\phi}$ and $\overline{\phi'}$ are multi-exact by construction, a natural transformation $\overline{\lambda}$ as in the statement is uniquely determined by the property that the pair $(\lambda,\overline{\lambda})$ defines a $2$-morphism, provided it exists.
	In order to show the existence of $\overline{\lambda}$, it suffices to apply \Cref{prop:lift-multiex} to the induced multi-exact morphism of abelian representations
	\begin{equation*}
		((\phi \xrightarrow{\lambda} \phi'), (\Phi \xrightarrow{\Lambda} \Phi'); (\kappa,\kappa')): (\Dcal_1,\Acal_1;\beta_1) \times \dots \times (\Dcal_n,\Acal_n;\beta_n) \rightarrow (\Dcal^{\Ical},\Acal^{\Ical};\beta^{\Ical})
	\end{equation*}
	and repeat the argument in the proof of \Cref{prop_nat-lift}. We leave the details to the interested reader.
\end{proof}

\begin{rem}\label{rem:lift-multinat}
	The analogue of \Cref{rem_comp-nat-lift} applies, namely:
	\begin{enumerate}
		\item While the existence of the natural transformation $\overline{\lambda}: \overline{\phi} \rightarrow \overline{\phi'}$ follows from the existence of $\Lambda$ (and from its compatibility with $\lambda$), the actual expression of $\overline{\lambda}$ only depends on $\lambda$.
		\item As a consequence of the uniqueness in the statement of \Cref{prop:lift-multinat}, we see that the assignment $(\lambda,\Lambda) \mapsto \bar{\lambda}$ is compatible with composition of $2$-morphisms in the following sense: Given, for $i = 1,2,3$, a multi-exact morphism of abelian representations 
		\begin{equation*}
			\alpha_i = (\phi_i,\Phi_i;\theta_i): (\Dcal_1,\Acal_1;\beta_1) \times \dots \times (\Dcal_n,\Acal_n;\beta_n) \rightarrow (\Dcal,\Acal;\beta)
		\end{equation*} 
		and given, for $i = 1,2$, a $2$-morphism $(\lambda_i,\Lambda_i): \alpha_i \rightarrow \alpha_{i+1}$, we have the equality of natural transformations of functors $\A(\beta_1) \rightarrow \Acal(\beta_2)$
		\begin{equation*}
			\overline{\lambda_2 \circ \lambda_1} = \overline{\lambda_2} \circ \overline{\lambda_1}.
		\end{equation*}
		\item As a consequence of the previous two points, we deduce that the natural transformation $\overline{\lambda}: \overline{\phi} \rightarrow \overline{\phi'}$ is invertible as soon as $\lambda: \phi \rightarrow \phi'$ is invertible.
	\end{enumerate}
\end{rem}

\section{Lifting abelian fibered categories}\label{sect:lift-fib}

In this section, we apply the previous results to the setting of abelian fibered categories. In the first place, we show how the results of \Cref{sect:lift-exfun} and \Cref{sect:lift-nat} allow one to lift fibered categories to the level of universal abelian factorizations in a natural way. We then derive some useful applications to cofinality properties of universal abelian factorizations indexed by a cofiltered category.

\subsection{Main result}

Throughout this section, we let $\Scal$ be a fixed small category. The notion of fibered category over $\Scal$ that it is natural to consider here corresponds to the notion of "catégorie clivée normalisée" in the original reference \cite[Exp.~VI, Defn.~7.1]{SGA1}. In detail:

\begin{defn}\label{defn:Scal-fib}
	\begin{enumerate}
		\item An \textit{(additive) $\Scal$-fibered category} $\Acal$ is the datum of
		\begin{itemize}
			\item for every $S \in \Scal$, an additive category $\Acal(S)$,
			\item for every morphism $f: T \rightarrow S$ in $\Scal$, an additive functor
			\begin{equation*}
				f^*: \Acal(S) \rightarrow \Acal(T),
			\end{equation*}
			called the \textit{inverse image functor} along $f$,
			\item for every pair of composable morphisms $f: T \rightarrow S$ and $g: S \rightarrow V$ in $\Scal$, a natural isomorphism of functors $\Acal(V) \rightarrow \Acal(T)$
			\begin{equation*}
				\conn = \conn_{f,g}: (gf)^* A \xrightarrow{\sim} f^* g^* A
			\end{equation*}
			called the \textit{connection isomorphism} at $(f,g)$
		\end{itemize}
		such that the following conditions are satisfied:
		\begin{enumerate}
			\item[($\Scal$-fib-0)] For every $S \in \Scal$, we have $\id_S^* = \id_{\Hbb(S)}$.
			\item[($\Scal$-fib-1)] For every triple of composable morphisms $f: T \rightarrow S$, $g: S \rightarrow V$ and $h: V \rightarrow W$ in $\Scal$, the diagram of functors $\Acal(W) \rightarrow \Acal(T)$
			\begin{equation*}
				\begin{tikzcd}
					(hgf)^* A \arrow{rr}{\conn_{f,hg}} \arrow{d}{\conn_{gf,h}} && f^* (hg)^* \arrow{d}{\conn_{g,h}} A \\
					(gf)^* h^* A \arrow{rr}{\conn_{f,g}} && f^* g^* h^* A
				\end{tikzcd}
			\end{equation*}
			is commutative.
		\end{enumerate}
	\item An additive $\Scal$-fibered category $\Acal$ as above is called \textit{abelian} if all additive categories $\Acal(S)$ are abelian and all inverse image functors $f^*$ are exact.
	\end{enumerate}
\end{defn}

\begin{rem}
	One could avoid to impose axiom ($\Scal$-fib-0) in such a strict form. However, it is well-known that strict unitality of fibered categories can always be enforced; see for example \cite[Lemma~2.5]{DelVoe}.
\end{rem}

In order to explain how to lift abelian $\Scal$-fibered categories, we need to recall the notion of morphism between abelian $\Scal$-fibered categories in the following form:

\begin{defn}\label{defn:mor-Scal-fib}
	Let $\Ccal_1$ and $\Ccal_2$ be two (additive) $\Scal$-fibered categories.
	\begin{enumerate}
		\item A \textit{$1$-morphism of (additive) $\Scal$-fibered categories} $R: \Ccal_1 \rightarrow \Ccal_2$ is the datum of
		\begin{itemize}
			\item for every $S \in \Scal$, an additive functor $R_S: \Ccal_1(S) \rightarrow \Ccal_2(S)$,
			\item for every morphism $f: T \rightarrow S$ in $\Scal$, a natural isomorphism of functors $\Ccal_1(S) \rightarrow \Ccal_2(T)$
			\begin{equation*}
				\theta = \theta_f: f^* R_S(A) \xrightarrow{\sim} R_T (f^* A),
			\end{equation*}
			called the \textit{$R$-transition isomorphism} along $f$
		\end{itemize}
		such that the following condition is satisfied:
		\begin{enumerate}
			\item[(mor-$\Scal$-fib)] For every pair of composable morphisms $f: T \rightarrow S$ and $g: S \rightarrow V$ in $\Scal$, the diagram of functors $\Ccal_1(V) \rightarrow \Ccal_2(T)$
			\begin{equation*}
				\begin{tikzcd}
					(gf)^* R_V(A) \arrow{rr}{\theta_{gf}} \arrow{d}{\conn^{(2)}_{f,g}} && R_T ((gf)^* A) \arrow{d}{\conn^{(1)}_{f,g}} \\
					f^* g^* R_V(A) \arrow{r}{\theta_g} & f^* R_S (g^* A) \arrow{r}{\theta_f} & R_T (f^* g^* A)
				\end{tikzcd}
			\end{equation*}
			is commutative.
		\end{enumerate}
	
		In case the $\Scal$-fibered categories $\Ccal_1$ and $\Ccal_2$ are abelian, we say that a morphism of (additive) $\Scal$-fibered categories $R: \Ccal_1 \rightarrow \Ccal_2$ is \textit{exact} if all functors $R_S$ are exact.
		\item Let $R_1, R_2: \Ccal_1 \rightarrow \Ccal_2$ be two morphisms of (additive) $\Scal$-fibered categories. A \textit{$2$-morphism of (additive) $\Scal$-fibered categories} $\gamma: R_1 \rightarrow R_2$ is the datum of
		\begin{itemize}
			\item for every $S \in \Scal$, a natural transformation of functors $\Ccal_1(S) \rightarrow \Ccal_2(S)$
			\begin{equation*}
				\gamma_S: R_1(C) \rightarrow R_2(C)
			\end{equation*}
		\end{itemize}
	    satisfying the following condition:
	    \begin{enumerate}
	    	\item[($2$-mor-$\Scal$-fib)] For every arrow $f: T \rightarrow S$ in $\Scal$, the diagram of functors $\Ccal_1(S) \rightarrow \Ccal_2(T)$
	    	\begin{equation*}
	    		\begin{tikzcd}
	    			f^* R_{1,S}(C) \arrow{r}{\gamma_S} \arrow{d}{\theta_f^{(1)}} & f^* R_{2,S}(C) \arrow{d}{\theta_f^{(2)}} \\
	    			R_{1,T}(f^* C) \arrow{r}{\gamma_T} & R_{2,T}(f^* C)
	    		\end{tikzcd}
	    	\end{equation*}
    	    is commutative.
	    \end{enumerate}
        A $2$-morphism of $\Scal$-fibered categories $\gamma$ is called a \textit{$2$-isomorphism} if all natural transformations $\gamma_S$ are invertible.
	\end{enumerate}
\end{defn}

\begin{rem}
	There exists an obvious notion of composition between morphisms of $\Scal$-fibered categories. In detail, consider three (additive) $\Scal$-fibered categories $\Ccal_1$, $\Ccal_2$ and $\Ccal_3$ and two morphisms $R_1: \Ccal_1 \rightarrow \Ccal_2$ and $R_2: \Ccal_2 \rightarrow \Ccal_3$; for every morphism $f: T \rightarrow S$ in $\Scal$, let $\theta_f^{(1)}$ and $\theta_f^{(2)}$ denote the corresponding transition isomorphisms along $f$. We define the \textit{composite morphism} $R_2 \circ R_1: \Ccal_1 \rightarrow \Ccal_3$ by assigning
	\begin{itemize}
		\item to every $S \in \Scal$ the composite functor
		\begin{equation*}
			(R_2 \circ R_1)_S: \Ccal_1(S) \xrightarrow{R_{1,S}} \Ccal_2(S) \xrightarrow{R_{2,S}} \Ccal_3(S),
		\end{equation*}
	    \item to every arrow $f: T \rightarrow S$ in $\Scal$, the natural isomorphism of functors $\Ccal_1(S) \rightarrow \Ccal_3(T)$
	    \begin{equation*}
	    	\theta_f^{(3)}: f^* (R_2 \circ R_1)_S(C) := f^* R_{2,S} R_{1,S}(C) \xrightarrow{\theta_f^{(2)}} R_{2,T}(f^* R_{1,S}(C)) \xrightarrow{\theta_f^{(1)}} R_{2,T} R_{1,T}(f^* C) =: (R_2 \circ R_1)_T(f^* C).
	    \end{equation*} 
	\end{itemize} 
    It is easy to check that these data satisfy condition (mor-$\Scal$-fib) and so define indeed a morphism of (additive) $\Scal$-fibered categories. We leave the details to the interested reader.
\end{rem}

With this notions at our disposal, we can define $\Scal$-fibered representations as follows:

\begin{defn}\label{defn:Sfibrepr}
	\begin{enumerate}
		\item We call \textit{(additive) $\Scal$-fibered representation} a triple $(\Ccal,\Ecal;R)$ consisting of two (additive) $\Scal$-fibered categories $\Ccal, \Ecal$ and a morphism of (additive) $\Scal$-fibered categories $R: \Ccal \rightarrow \Ecal$. 
		\item We say that the $\Scal$-fibered representation $(\Ccal,\Ecal;R)$ is \textit{abelian} if the $\Scal$-fibered category $\Ecal$ is abelian in the sense of \Cref{defn:Scal-fib}.
	\end{enumerate}
\end{defn}

\begin{defn}
	For $i = 1,2$, let $(\Ccal_i,\Ecal_i;R_i)$ be an $\Scal$-fibered representation in the sense of \Cref{defn:Sfibrepr}.
	\begin{enumerate}
		\item We call \textit{morphism of $\Scal$-fibered representations} from $(\Ccal_1,\Ecal_1;R_1)$ to $(\Ccal_2,\Ecal_2;R_2)$ a triple $\alpha = (F,G;\kappa)$ consisting of two additive morphisms of $\Scal$-fibered categories
		\begin{equation*}
			F: \Ccal_1 \rightarrow \Ccal_2, \qquad G: \Ecal_1 \rightarrow \Ecal_2
		\end{equation*}
		and a $2$-isomorphism between morphisms of $\Scal$-fibered categories $\Ccal_1 \rightarrow \Ecal_2$
		\begin{equation*}
			\kappa: G \circ R_1 \xrightarrow{\sim} R_2 \circ F.
		\end{equation*}
	    \item In case the additive $\Scal$-fibered representations $(\Ccal_1,\Ecal_1;\kappa_1)$ and $(\Ccal_2,\Ecal_2;\kappa_2)$ are abelian, we say that a $1$-morphism $\alpha: (\Ccal_1,\Ecal_1;\kappa_1) \rightarrow (\Ccal_2,\Ecal_2;\kappa_2)$ is \textit{exact} if the additive morphism of $\Scal$-fibered categories $G: \Ecal_1 \rightarrow \Ecal_2$ is exact.
	\end{enumerate} 
\end{defn}

\begin{rem}\label{rem-morSfib-App} 
	Requiring the commutativity condition of axiom (mor-$\Scal$-fib) in \Cref{defn:mor-Scal-fib} is the same as asking that, for every pair of composable morphisms $f: T \rightarrow S$ and $g: S \rightarrow V$ in $\Scal$, the connection isomorphisms of $\Ccal$ and $\Ecal$ at $(f,g)$ define a $2$-morphism between $1$-morphisms of representations $(\Ccal(V),\Ecal(V);R_V) \rightarrow (\Ccal(T),\Ecal(T);R_T)$
	\begin{equation*}
		(\conn_{f,g},\conn_{f,g}): ((gf)^*,(gf)^*;\theta_{gf}) \rightarrow (f^*, f^*;\theta_f) \circ (g^*,g^*;\theta_g).
	\end{equation*}
\end{rem}

In the first place, we show that the universal abelian factorizations arising from an abelian $\Scal$-fibered representation canonically assemble into a similar abelian $\Scal$-fibered representation:

\begin{prop}\label{prop_lifting-fib}
	Let $(\Dcal,\Acal;\beta)$ be an abelian $\Scal$-fibered representation. Then:
	\begin{enumerate}
		\item Associating
		\begin{itemize}
			\item to every $S \in \Scal$, the abelian category $\A(\beta_S)$,
			\item to every morphism $f: T \rightarrow S$ in $\Scal$, the exact functor
			\begin{equation*}
				\overline{f^*}: \A(\beta_S) \rightarrow \A(\beta_T)
			\end{equation*}
			obtained via \Cref{prop_lift-exfunct}(1),
			\item to every pair of composable morphisms $f: T \rightarrow S$ and $g: S \rightarrow V$ in $\Scal$, the natural isomorphism of functors $\A(\beta_V) \rightarrow \A(\beta_T)$
			\begin{equation}\label{conn:lift}
				\overline{\conn} = \overline{\conn}_{f,g}: \overline{(gf)^*} \xrightarrow{\sim} \overline{f^*} \overline{g^*}
			\end{equation}
			obtained via \Cref{prop_nat-lift}
		\end{itemize} 
		defines an abelian $\Scal$-fibered category $\A(\beta)$. 
		\item Associating
		\begin{itemize}
			\item to every $S \in \Scal$, the faithful exact functor $\iota_S: \A(\beta_S) \hookrightarrow \Acal(S)$,
			\item to every morphism $f: T \rightarrow S$ in $\Scal$, the natural isomorphism of functors $\A(\beta_S) \rightarrow \Acal(T)$
			\begin{equation}\label{theta:lift}
				\tilde{\theta} = \tilde{\theta}_f: \overline{f^*} \iota_S(M) \xrightarrow{\sim} \iota_T(\overline{f^*} M)
			\end{equation}
			obtained via \Cref{prop_lift-exfunct}(2)
		\end{itemize}
		defines an exact morphism of abelian $\Scal$-fibered categories $\iota: \A(\beta) \hookrightarrow \Acal$.
	\end{enumerate}
\end{prop}
\begin{proof}
	Recall that the constructions of \Cref{prop_lift-exfunct} and \Cref{prop_nat-lift} are compatible with composition by \Cref{rem_comp-lift}(2) and \Cref{rem_comp-nat-lift}(2), respectively. 
	As a consequence, the isomorphisms \eqref{conn:lift} necessarily satisfy conditions ($\Scal$-fib-0) and ($\Scal$-fib-1) of \Cref{defn:Scal-fib}, which implies (1). Similarly, the natural isomorphisms \eqref{theta:lift} satisfy condition ($m$-ETS) of \Cref{defn:ETS}, which implies (2). 
\end{proof}

\begin{rem}
	In fact, as $S$ varies in $\Scal$, the abelian representations $(\Dcal(S),\A(\beta_S);\pi_S)$ and $(\A(\beta_S),\Acal(S);\iota_S)$ assemble into abelian $\Scal$-fibered representations.
\end{rem}

Secondly, for sake of completeness, we also prove an enhanced version of \Cref{prop_lift-exfunct} in the $\Scal$-fibered setting:

\begin{prop}\label{prop:lift-morfib}
	For $i = 1,2$, let $(\Dcal_i,\Acal_i;\beta_i)$ be an abelian $\Scal$-fibered representation. Suppose that we are given an exact morphism of $\Scal$-fibered representations 
	\begin{equation*}
		(\phi,\Phi;\kappa): (\Dcal_1,\Acal_1;\beta_1) \rightarrow (\Dcal_2,\Acal_2;\beta_2).
	\end{equation*}
	Then, associating 
	\begin{itemize}
		\item to every $S \in \Scal$, the exact functor
		\begin{equation*}
			\overline{\phi_S}: \A(\beta_{1,S}) \rightarrow \A(\beta_{2,S})
		\end{equation*}
		obtained via \Cref{prop_lift-exfunct},
		\item to every morphism $f: T \rightarrow S$ in $\Scal$, the natural isomorphism of functors $\A(\beta_{1,S}) \rightarrow \A(\beta_{2,T})$
		\begin{equation}\label{tilde:kappa-fib}
			\tilde{\kappa} = \tilde{\kappa}_f: \overline{f^*} \overline{\phi_{1,S}}(M) \xrightarrow{\sim} \overline{\phi_{2,T}}(\overline{f^*} M)
		\end{equation}
		obtained via \Cref{prop_nat-lift}
	\end{itemize}
	defines an exact $1$-morphism of $\Scal$-fibered categories $\A(\beta_1) \rightarrow \A(\beta_2)$.
\end{prop}
\begin{proof}
	Arguing as in the proof of the previous result, we see that the isomorphisms \eqref{tilde:kappa-fib} necessarily satisfy condition (mor-$\Scal$-fib), which implies the thesis. 
\end{proof}

\begin{rem}
	In fact, as $S$ varies in $\Scal$, the $1$-morphisms of abelian representations $(\phi_S,\overline{\phi_S};\id)$ and $(\overline{\phi_S},\Phi_S;\tilde{\kappa}_S)$ assemble into $1$-morphisms of abelian $\Scal$-fibered representations.
\end{rem}

\subsection{Cofinality properties}

We now specialize the discussion of the previous subsection to the case where the base category is co-filtered; for sake of clarity, throughout this subsection we let $\Ucal$ denote a co-filtered small category. Our goal is to study filtered $2$-colimits of universal abelian factorizations arising from some abelian $\Ucal$-fibered category.

Recall that, for every (additive) $\Ucal$-fibered category $\Ccal$, it makes sense to define the filtered $2$-colimit
\begin{equation}\label{2-colim:Ccal}
	2-\varinjlim_{U \in \Ucal^{op}} \Ccal(U),
\end{equation}  
which is a well-behaved (additive) category whose objects and morphisms admit the usual explicit description in terms of equivalence classes of representatives. 
If $\Ccal$ is an abelian $\Ucal$-fibered category in the sense of \Cref{defn:Scal-fib}, the $2$-colimit \eqref{2-colim:Ccal} is an abelian category. 

\begin{lem}\label{lem_A-2colim}
	Let $\Dcal$ be an additive $\Ucal$-fibered category. Then the following statements hold:
	\begin{enumerate}
		\item There exists a canonical equivalence
		\begin{equation}\label{A:2-colim}
			2-\varinjlim_{U \in \Ucal^{op}} \A(\Dcal(U)) \xrightarrow{\sim} \A(2-\varinjlim_{U \in \Ucal^{op}} \Dcal(U))
		\end{equation}
	    making the diagram
	    \begin{equation*}
	    	\begin{tikzcd}
	    		& 2-\varinjlim_{U \in \Ucal^{op}} \Dcal(U) \arrow{dl} \arrow{dr} \\
	    		2-\varinjlim_{U \in \Ucal^{op}} \A(\Dcal(U)) \arrow[dashed]{rr}{\sim} && \A(2-\varinjlim_{U \in \Ucal^{op}} \Dcal(U))
	    	\end{tikzcd}
	    \end{equation*}
        commute.
        \item More generally, for every abelian $\Ucal$-fibered representation $(\Dcal,\Acal;\beta)$, there exists a canonical equivalence
        \begin{equation}\label{A(beta):2-colim}
        	2-\varinjlim_{U \in \Ucal^{op}} \A(\beta_U) \xrightarrow{\sim} \A(2-\varinjlim_{U \in \Ucal^{op}} \beta_U) 
        \end{equation}
        making the whole diagram
        \begin{equation*}
        	\begin{tikzcd}
        		& 2-\varinjlim_{U \in \Ucal^{op}} \Dcal(U) \arrow{dl} \arrow{dr} \\
        		2-\varinjlim_{U \in \Ucal^{op}} \A(\beta_U) \arrow[dashed]{rr}{\sim} \arrow{dr} && \A(2-\varinjlim_{U \in \Ucal^{op}} \beta_U) \arrow{dl} \\
        		& 2-\varinjlim_{U \in \Ucal^{op}} \Acal(U)
        	\end{tikzcd}
        \end{equation*}
        commute.
	\end{enumerate}
\end{lem}
\begin{proof}
	\begin{enumerate}
		\item The functor \eqref{A:2-colim} is induced by the $\Ucal^{op}$-indexed system of exact functors $\A(\Dcal(U)) \rightarrow \A(2-\varinjlim_{U \in \Ucal^{op}} \Dcal(U))$ obtained by applying \Cref{thm_A(D)}(2)(i) to the obvious $1$-morphisms of abelian representations
		\begin{equation*}
			(\Dcal(U),\Acal(U);\beta_U) \rightarrow (2-\varinjlim_{U \in \Ucal^{op}} \Dcal(U), 2-\varinjlim_{U \in \Ucal^{op}} \Acal(U); 2-\varinjlim_{U \in \Ucal^{op}} \beta_U), \qquad U \in \Ucal.
		\end{equation*} 
	    A canonical quasi-inverse to it is determined by the universal property of $\A(2-\varinjlim_{U \in \Ucal^{op}} \Dcal(U))$ via \Cref{thm_A(D)}(2)(i).
	    \item As for the previous point, the functor \eqref{A(beta):2-colim} is induced by applying \Cref{prop_lift-exfunct}(1) for each $U \in \Ucal$ and passing to the $2$-colimit, while a canonical quasi-inverse is determined by the universal property of $\A(2-\varinjlim_{U \in \Ucal^{op}} \beta_U)$ via \Cref{prop_lift-exfunct}(1).
	\end{enumerate}
\end{proof}

We now derive an consequence of the previous result. In order to formulate it, we need a suitable notion of denseness for (additive) $\Ucal$-fibered categories:

\begin{defn}\label{defn:w-cof}
	Let $\Ccal$ be an (additive) $\Ucal$-fibered category, and let $\Ccal^0$ be an (additive) $\Ucal$-fibered subcategory of $\Ccal$. We say that $\Ccal^0$ is \textit{weakly dense} in $\Ccal$ if the canonical functor
	\begin{equation*}
		2-\varinjlim_{U \in \Ucal^{op}} \Ccal^0(U) \rightarrow 2-\varinjlim_{U \in \Ucal^{op}} \Ccal(U) 
	\end{equation*}
    is an equivalence.
\end{defn}

\begin{rem}
	Concretely, saying that $\Ccal^0$ is weakly dense in $\Ccal$ amounts to requiring the following two conditions:
	\begin{enumerate}
		\item For every $U \in \Ucal$ and every object $C \in \Ccal(U)$, there exists an arrow $v: V \rightarrow U$ in $\Ucal$ (depending on $C$) such that the object $v^* C \in \Ccal(V)$ lies in the essential image of $\Ccal^0(V)$.
		\item For every $V \in \Ucal$, every pair of objects $C, C' \in \Ccal(U)$, and every morphism $c: C \rightarrow C'$ in $\Ccal(V)$, there exists an arrow $w: W \rightarrow V$ in $\Ucal$ (depending on $c$) such that the morphism $w^* c: w^* C \rightarrow w^* C'$ lies in $\Ccal^0(W)$.
	\end{enumerate}
\end{rem}

Note that, given an abelian representation $(\Dcal,\Acal;\beta)$ and a (possibly non-full) additive subcategory $\Dcal^0$ of $\Dcal$, the inclusion functor $\Dcal^0 \hookrightarrow \Dcal$ defines a canonical exact $1$-morphism of abelian representations
\begin{equation*}
	(\Dcal^0,\Acal;\beta^0) \rightarrow (\Dcal,\Acal;\beta),
\end{equation*}
where we have set $\beta^0 := \beta|_{\Dcal^0}$. Applying \Cref{prop_lift-exfunct}, this yields a canonical exact functor
\begin{equation*}
	\A(\beta^0) \rightarrow \A(\beta)
\end{equation*}
which is faithful by construction (but a priori not full nor essentially surjective, even if one assumes that the inclusion $\Dcal^0 \hookrightarrow \Dcal$ is just full or just essentially surjective, respectively). Thus we may naturally regard $\A(\beta^0)$ as an abelian subcategory of $\A(\beta)$. 

As a consequence of \Cref{prop:lift-morfib}, the analogous result holds for $\Ucal$-fibered representations: in this case, one starts from an abelian $\Ucal$-fibered representation $(\Dcal,\Acal;\beta)$ together with a (possibly non-full) $\Ucal$-fibered subcategory $\Dcal^0$ of $\Dcal$, and one gets a canonical faithful exact morphism of $\Ucal$-fibered categories
\begin{equation*}
	\A(\beta^0) \rightarrow \A(\beta)
\end{equation*}
expressing $\A(\beta^0)$ naturally as a $\Ucal$-fibered subcategory of $\A(\beta)$.

As a consequence of the previous result, we get the following:

\begin{cor}\label{cor:lift-weakly-cofinal}
	Let $(\Dcal,\Acal;\beta)$ be an abelian $\Ucal$-fibered representation; in addition, let $\Dcal^0$ be a (possibly non-full) additive $\Ucal$-fibered subcategory of $\Dcal$, and let $(\Dcal^0,\Acal;\beta^0)$ denote the restricted abelian $\Ucal$-fibered representation.
	Suppose that $\Dcal^0$ is weakly cofinal in $\Dcal$ in the sense of \Cref{defn:w-cof}. Then $\A(\beta^0)$ is weakly dense in $\A(\beta)$ as well.
\end{cor}
\begin{proof}
	Consider the canonical commutative diagram of abelian categories and exact functors
	\begin{equation*}
		\begin{tikzcd}
			2-\varinjlim_{U \in \Ucal^{op}} \A(\beta^{0}_U) \arrow{rr} \isoarrow{d} && 2-\varinjlim_{U \in \Ucal^{op}} \A(\beta_U) \isoarrow{d} \\
			\A(2-\varinjlim_{U \in \Ucal^{op}} \beta^{0}_U) \arrow{rr} && \A(2-\varinjlim_{U \in \Ucal^{op}} \beta_U)
		\end{tikzcd}
	\end{equation*}
	in which the two vertical arrows are the equivalences of \Cref{lem_A-2colim}. Proving the thesis amounts to showing that the upper vertical arrow is an equivalence, and to this end it suffices to show that the lower horizontal arrow is an equivalence. But, since $\Dcal^{0}$ is weakly cofinal in $\Dcal$ by hypothesis, the obvious $1$-morphism of abelian representations
	\begin{equation*}
		(2-\varinjlim_{U \in \Ucal^{op}} \Dcal^{0}(U), 2-\varinjlim_{U \in \Ucal^{op}} \Acal(U); 2-\varinjlim_{U \in \Ucal^{op}} \beta^{0}_U) \rightarrow (2-\varinjlim_{U \in \Ucal^{op}} \Dcal(U), 2-\varinjlim_{U \in \Ucal^{op}} \Acal(U); 2-\varinjlim_{U \in \Ucal^{op}} \beta_U)
	\end{equation*}
	is an equivalence (in the sense of \Cref{defn_adj-repr}). Therefore, by \Cref{cor_lift-equiv}, the lower horizontal arrow in the above diagram is an equivalence as well.
\end{proof}

\section{Lifting (external) tensor structures}\label{sect:lift-ETS}

In this final section, we explain how to lift the structure of monoidal fibered categories to the level of universal abelian factorizations; we achieve this by combining the results of \Cref{sect:multi} on multi-exact functors with those of \Cref{sect:lift-fib} on abelian fibered categories. 

Recall that in \cite{Ter23Fib} we have shown how the usual structure and properties of monoidal fibered categories, which are usually expressed in terms of the internal tensor product, can be canonically rephrased in terms of the external tensor product - provided the base category admits binary products, so that the external tensor product can be actually defined.

In view of our applications to perverse Nori motives, where working with the external tensor product is the most natural choice, we chose to formulate the results of the present section in terms the external tensor product. The interested reader can easily translate our statements and proofs to the usual setting via the dictionary of \cite{Ter23Fib}. 

\subsection{Main result}

As in the previous section, we work over a fixed small category $\Scal$, but we now assume that $\Scal$ admits binary products.

For convenience, we start by recalling the notion of external tensor structure on $\Scal$-fibered categories from \cite[\S\S~2, 8]{Ter23Fib}, adapted to the additive and abelian settings:

\begin{defn}\label{defn:ETS}
	Let $\Ccal$ be an additive $\Scal$-fibered category.
	\begin{enumerate}
		\item An \textit{external tensor structure} $(\boxtimes,m)$ on $\Ccal$ is the datum of
		\begin{itemize}
			\item for every $S_1, S_2 \in \Scal$, a bi-additive functor
			\begin{equation*}
				-\boxtimes- = - \boxtimes_{S_1,S_2} -: \Ccal(S_1) \times \Ccal(S_2) \rightarrow \Ccal(S_1 \times S_2),
			\end{equation*}
			called the \textit{external tensor product functor} over $(S_1,S_2)$,
			\item for every choice of morphisms $f_i: T_i \rightarrow S_i$ in $\Scal$, $i = 1,2$, a natural isomorphism of functors $\Ccal(S_1) \times \Ccal(S_2) \rightarrow \Ccal(T_1 \times T_2)$
			\begin{equation*}
				m = m_{f_1,f_2}: f_1^* A_1 \boxtimes f_2^* A_2 \xrightarrow{\sim} (f_1 \times f_2)^*(A_1 \boxtimes A_2),
			\end{equation*}
			called the \textit{external monoidality isomorphism} along $(f_1,f_2)$
		\end{itemize}
		satisfying the following condition:
		\begin{enumerate}
			\item[($m$ETS)] For every choice of composable morphisms $f_i: T_i \rightarrow S_i$, $g_i: S_i \rightarrow V_i$, $i = 1,2$, the diagram of functors $\Ccal(V_1) \times \Ccal(V_2) \rightarrow \Ccal(T_1 \times T_2)$
			\begin{equation*}
				\begin{tikzcd}[font=\small]
					(g_1 f_1)^* C_1 \boxtimes (g_2 f_2)^* C_2  \arrow{rr}{m} \arrow[equal]{d} && (g_1 f_1 \times g_2 f_2)^* (C_1 \boxtimes C_2)  \arrow[equal]{d} \\
					f_1^* g_1^* C_1 \boxtimes f_2^* g_2^* C_2  \arrow{r}{m} & (f_1 \times f_2)^* (g_1^* C_1 \boxtimes g_2^* C_2) \arrow{r}{m} & (f_1 \times f_2)^* (g_1 \times g_2)^* (C_1 \boxtimes C_2)
				\end{tikzcd}
			\end{equation*}	
			is commutative.
		\end{enumerate}
	    \item In case the $\Scal$-fibered category $\Ccal$ is abelian, we say that an external tensor structure $(\boxtimes,m)$ on $\Ccal$ is \textit{exact} if all external tensor product functors $- \boxtimes_S -$ are bi-exact.
	\end{enumerate}
\end{defn}

In order to extend the notion of external tensor structure to $\Scal$-fibered representations, we also need to recall the notion of external tensor structure on a morphism of $\Scal$-fibered categories from \cite[\S~8]{Ter23Fib}, adapted to the additive setting:

\begin{defn}\label{defn:ETS-mor}
	Let $R: \Ccal \rightarrow \Ecal$ be a $1$-morphism of (additive) $\Scal$-fibered categories. 
	An \textit{external tensor structure} on $R$ (with respect to $(\boxtimes_{\Ccal},m_{\Ccal})$ and $(\boxtimes_{\Ecal},m_{\Ecal})$) is the datum of
	\begin{itemize}
		\item for every $S_1, S_2 \in \Scal$, a natural isomorphism of functors $\Ccal(S_1) \times \Ccal(S_2) \rightarrow \Ecal(S_1 \times S_2)$
		\begin{equation*}
			\rho = \rho_{S_1,S_2}: R_{S_1}(C_1) \boxtimes_{\Ecal} R_{S_2}(C_2) \xrightarrow{\sim} R_{S_1 \times S_2}(C_1 \boxtimes_{\Ccal} C_2),
		\end{equation*}
		called the \textit{external $R$-monoidality isomorphism} at $(S_1,S_2)$,
	\end{itemize}
	satisfying the following condition:
	\begin{enumerate}
		\item[(mor-ETS)] For every two morphisms $f_i: T_i \rightarrow S_i$ in $\Scal$, $i = 1,2$, the natural diagram of functors $\Hbb_1(S_1) \times \Hbb_1(S_2) \rightarrow \Hbb_2(T_1 \times T_2)$
		\begin{equation*}
			\begin{tikzcd}
				f_1^* R_{S_1}(C_1) \boxtimes_{\Ecal} f_2^* R_{S_2}(C_2) \arrow{r}{m_2} \arrow{d}{\theta} & (f_1 \times f_2)^* (R_{S_1}(C_1) \boxtimes_{\Ecal} R_{S_2}(C_2)) \arrow{r}{\rho} & (f_1 \times f_2)^* R_{S_1 \times S_2}(C_1 \boxtimes_{\Ccal} C_2) \arrow{d}{\theta} \\
				R_{T_1}(f_1^* C_1) \boxtimes_{\Ecal} R_{T_2}(f_2^* C_2) \arrow{r}{\rho} & R_{T_1 \times T_2}(f_1^* C_1 \boxtimes_{\Ccal} f_2^* C_2) \arrow{r}{m_1} & R_{T_1 \times T_2}((f_1 \times f_2)^* (C_1 \boxtimes_{\Ccal} C_2))
			\end{tikzcd}
		\end{equation*}
		is commutative.
	\end{enumerate}
\end{defn}

With these notions at our disposal, we can define external tensor structures on additive $\Scal$-fibered representations in a natural way as follows:

\begin{defn}
	Let $(\Ccal,\Ecal;\beta)$ be an (additive) $\Scal$-fibered representation. 
	\begin{enumerate}
		\item An \textit{external tensor structure} on $(\Ccal,\Ecal;\beta)$ is the datum of
		\begin{itemize}
			\item an external tensor structure $(\boxtimes_{\Ccal},m_{\Ccal})$ on $\Ccal$,
			\item an external tensor structure $(\boxtimes_{\Ecal},m_{\Ecal})$ on $\Ecal$,
			\item an external tensor structure $\rho$ on the morphism of $\Scal$-fibered categories $\beta: \Ccal \rightarrow \Ecal$.
		\end{itemize}
		We write it as a triple $(\boxtimes_{\Ccal},\boxtimes_{\Ecal};\rho)$ where, for convenience, we leave the symbols $m_{\Ccal}$ and $m_{\Ecal}$ unexplicit.
		\item In case the $\Scal$-fibered representation $(\Ccal,\Ecal;\beta)$ is abelian, we say that an external tensor structure on $(\Ccal,\Ecal;\beta)$ is \textit{exact} if the external tensor structure $(\boxtimes_{\Ecal},m_{\Ecal})$ on the abelian $\Scal$-fibered category $\Ecal$ is exact.
	\end{enumerate}
\end{defn}

In \cite[\S 2]{Ter23Fib} we explained how to axiomatize external tensor structures on $\Scal$-fibered categories as suitable $(\Scal^2 \times \Ical)$-fibered categories, where $\Ical$ denotes the category with two objects and exactly one non-trivial morphism between them. 

The result that we discuss now bears a strong resemblance to those concerning abelian fibered categories. Note that, strictly speaking, we are not allowed to deduce them directly from the previous results on fibered categories via the constructions of \cite[\S 2]{Ter23Fib}: the reason is that the external tensor product functors are not exact on the product abelian category but rather separately exact on each variable. However, the proof method is essentially the same as before.

\begin{prop}\label{prop:lift-ETS}
	Let $(\Dcal,\Acal;\beta)$ be an abelian $\Scal$-fibered representation, and suppose that we are given an exact external tensor structure $(\boxtimes_{\Dcal},\boxtimes_{\Acal};\rho)$ on it. Then:
	\begin{enumerate}
		\item Associating
		\begin{itemize}
			\item to every $S_1,S_2 \in \Scal$, the bi-exact functor
			\begin{equation*}
				- \overline{\boxtimes}_{S_1,S_2} - = - \overline{\boxtimes} -: \A(\beta_{S_1}) \times \A(\beta_{S_2}) \rightarrow \A(\beta_{S_1 \times S_2})
			\end{equation*}
			obtained via \Cref{prop:lift-multiex}(1),
			\item to every two morphisms $f_i: T_i \rightarrow S_i$, $i = 1,2$, the natural transformation of exact functors $\A(\beta_{S_1}) \times \A(\beta_{S_2}) \rightarrow \A(\beta_{T_1 \times T_2})$
			\begin{equation}\label{monoext:fib-lift}
				\tilde{m} = \tilde{m}_{f_1,f_2}: f_1^* M_1 \overline{\boxtimes} f_2^* M_2 \xrightarrow{\sim} (f_1 \times f_2)^* (M_1 \overline{\boxtimes} M_2)
			\end{equation}
			obtained via \Cref{prop:lift-multinat}
		\end{itemize}
		defines an exact external tensor structure $(\overline{\boxtimes},\tilde{m})$ on the abelian $\Scal$-fibered category $\A(\beta)$.
		\item Associating
		\begin{itemize}
			\item to every $S_1, S_2 \in \Scal$, the natural isomorphism of functors $\A(\beta_{S_1}) \times \A(\beta_{S_2}) \rightarrow \Acal(S_1 \times S_2)$
			\begin{equation}\label{tilde:rho-fib}
				\tilde{\rho} = \tilde{\rho}_{S_1,S_2}: \iota_{S_1}(M_1) \boxtimes \iota_{S_2}(M_2) \xrightarrow{\sim} \iota_{S_1 \times S_2}(M_1 \overline{\boxtimes} M_2)
			\end{equation}
			obtained via \Cref{prop:lift-multiex}(2)
		\end{itemize}
		defines an external tensor structure on the morphism of $\Scal$-fibered categories $\iota: \A(\beta) \hookrightarrow \Acal$ of \Cref{prop:lift-morfib} (with respect to the external tensor structure obtained in the previous point).
	\end{enumerate} 
\end{prop}
\begin{proof}
	Arguing as in the proofs of the previous results, we see that the natural isomorphisms \eqref{monoext:fib-lift} necessarily satisfy condition ($m$-ETS) of \Cref{defn:ETS}(1), which implies (1). Similarly, the natural isomorphisms \eqref{tilde:rho-fib} necessarily satisfy condition (mor-ETS) of \Cref{defn:ETS}(2), which implies (2).
\end{proof}

\subsection{Lifting associativity and commutativity constraints}

In \cite[\S\S~3,4]{Ter23Fib} we explained how to translate the usual notions of associativity and commutativity constraints for monoidal $\Scal$-fibered categories in the setting of external tensor structures. We now want to discuss briefly how to lift these constraints to universal abelian factorizations. As usual, we start by recalling the relevant definitions from \cite{Ter23Fib}, adapted to the additive setting:

\begin{defn}\label{defn:ETS-asso-symm}
	Let $\Ccal$ be an (additive) $\Scal$-fibered category endowed with an (additive) external tensor structure $(\boxtimes,m)$.
	\begin{enumerate}
		\item An \textit{external associativity constraint} $a$ on $(\boxtimes,m)$ is the datum of
		\begin{itemize}
			\item for every $S_1, S_2, S_3 \in \Scal$, a natural isomorphism of functors $\Ccal(S_1) \times \Ccal(S_2) \times \Ccal(S_3) \rightarrow \Ccal(S_1 \times S_2 \times S_3)$
			\begin{equation*}
				a = a_{S_1,S_2,S_3}: (A_1 \boxtimes A_2) \boxtimes A_3 \xrightarrow{\sim} A_1 \boxtimes (A_2 \boxtimes A_3)
			\end{equation*}
		\end{itemize}
		satisfying the following conditions:
		\begin{enumerate}
			\item[($a$ETS-1)] For every choice of $S_1, S_2, S_3, S_4 \in \Scal$, the diagram of functors $\Ccal(S_1) \times \Ccal(S_2) \times \Ccal(S_3) \times \Ccal(S_4) \rightarrow \Ccal(S_1 \times S_2 \times S_3 \times S_4)$
			\begin{equation*}
				\begin{tikzcd}
					((C_1 \boxtimes C_2) \boxtimes C_3) \boxtimes C_4 \arrow{d}{a} \arrow{rr}{a} && (C_1 \boxtimes C_2) \boxtimes (C_3 \boxtimes C_4) \arrow{d}{a} \\
					(C_1 \boxtimes (C_2 \boxtimes C_3)) \boxtimes C_4 \arrow{r}{a} & C_1 \boxtimes ((C_2 \boxtimes C_3) \boxtimes C_4) \arrow{r}{a} & C_1 \boxtimes (C_2 \boxtimes (C_3 \boxtimes C_4))
				\end{tikzcd}
			\end{equation*}
			is commutative.
			\item[($a$ETS-2)] For every three morphisms $f_i: T_i \rightarrow S_i$ in $\Scal$, $i = 1,2,3$, the diagram of functors $\Ccal(S_1) \times \Ccal(S_2) \times \Ccal(S_3) \rightarrow \Ccal(T_1 \times T_2 \times T_3)$
			\begin{equation*}
				\begin{tikzcd}
					(f_1^* C_1 \boxtimes f_2^* C_2) \boxtimes f_3^* C_3 \arrow{r}{m} \arrow{d}{a} & (f_1 \times f_2)^* (C_1 \boxtimes C_2) \boxtimes f_3^* C_3 \arrow{r}{m} & (f_1 \times f_2 \times f_3)^* ((C_1 \boxtimes C_2) \boxtimes C_3) \arrow{d}{a} \\
					f_1^* C_1 \boxtimes (f_2^* C_2 \boxtimes f_3^* C_3) \arrow{r}{m} & f_1^* C_1 \boxtimes (f_2 \times f_3)^* (C_2 \boxtimes C_3) \arrow{r}{m} & (f_1 \times f_2 \times f_3)^* (C_1 \boxtimes (C_2 \boxtimes C_3))  
				\end{tikzcd}
			\end{equation*}
			is commutative.
		\end{enumerate}
	    \item An \textit{external commutativity constraint} $c$ on $(\boxtimes,m)$ is the datum of
	    \begin{itemize}
	    	\item for every $S_1, S_2 \in \Scal$, a natural isomorphism of functors $\Ccal(S_1) \times \Ccal(S_2) \rightarrow \Ccal(S_1 \times S_2)$
	    	\begin{equation*}
	    		c = c_{S_1,S_2}: A_1 \boxtimes A_2 \xrightarrow{\sim} \tau^*(A_2 \boxtimes A_1),
	    	\end{equation*}
	    	(where $\tau: S_1 \times S_2 \rightarrow S_2 \times S_1$ denotes the permutation isomorphism)
	    \end{itemize}
	    satisfying the following conditions:
	    \begin{enumerate}
	    	\item[($c$ETS-1)]  For every $S_1, S_2 \in \Scal$, the diagram of functors $\Ccal(S_1) \times \Ccal(S_2) \rightarrow \Ccal(S_1 \times S_2)$
	    	\begin{equation*}
	    		\begin{tikzcd}
	    			C_1 \boxtimes C_2 \arrow{r}{c} \arrow[equal]{dr} & \tau^* (C_2 \boxtimes C_1) \arrow{d}{c} \\
	    			& \tau^* \tau^*(C_1 \boxtimes C_2)
	    		\end{tikzcd}
	    	\end{equation*}
	    	is commutative.
	    	\item[($c$ETS-2)] For every choice of of morphisms $f_i: T_i \rightarrow S_i$ in $\Scal$, $i = 1,2$, the diagram of functors $\Ccal(S_1) \times \Ccal(S_2) \rightarrow \Ccal(T_2 \times T_1)$
	    	\begin{equation*}
	    		\begin{tikzcd}
	    			f_1^* C_1 \boxtimes f_2^* C_2 \arrow{rr}{c} \arrow{d}{m} && \tau^* (f_2^* C_2 \boxtimes f_1^* C_1) \arrow{d}{m} \\
	    			(f_1 \times f_2)^* (C_1 \boxtimes C_2)\arrow{r}{c} & (f_1 \times f_2)^* \tau^* (C_2 \boxtimes C_1) \arrow[equal]{r} & \tau^* (f_2 \times f_1)^* (C_2 \boxtimes C_1)
	    		\end{tikzcd}
	    	\end{equation*}
	    	(where $\tau: S_1 \times S_2 \xrightarrow{\sim} S_2 \times S_1$ and $\tau: T_1 \times T_2 \xrightarrow{\sim} T_2 \times T_1$ denote the permutation isomorphisms) is commutative.
	    \end{enumerate}
	\end{enumerate}
\end{defn}

In order to define reasonable notions of associativity and commutativity constraints in the setting of $\Scal$-fibered representations, we also need to recall the notions of associativity and symmetry for external tensor structures on morphisms of $\Scal$-fibered categories from \cite[\S~8]{Ter23Fib}, adapted to the additive setting:

\begin{defn}\label{defn:mor-acETS}
	Let $\Ccal$ and $\Ecal$ be two (additive) $\Scal$-fibered categories endowed with additive external tensor structures $(\boxtimes_{\Ccal},m_{\Ccal})$ and $(\boxtimes_{\Ecal},m_{\Ecal})$, respectively; in addition, let $R: \Ccal \rightarrow \Ecal$ be an additive morphism of $\Scal$-fibered categories, and let $\rho$ be an external tensor structure on $R$ (with respect to $(\boxtimes_{\Ccal},m_{\Ccal})$ and $(\boxtimes_{\Ecal},m_{\Ecal})$).
	\begin{enumerate}
		\item Let $a_{\Ccal}$ and $a_{\Ecal}$ be external associativity constraints on $(\boxtimes_{\Ccal},m_{\Ccal})$ and $(\boxtimes_{\Ecal},m_{\Ecal})$, respectively. We say that $\rho$ is \textit{associative} (with respect to $a_{\Ccal}$ and $a_{\Ecal}$) if it satisfies the following additional condition:
		\begin{enumerate}
			\item[(mor-$a$ETS)] For every $S_1, S_2, S_3 \in \Scal$, the diagram of functors $\Ccal(S_1) \times \Ccal(S_2) \times \Ccal(S_3) \rightarrow \Ecal(S_1 \times S_2 \times S_3)$
			\begin{equation*}
				\begin{tikzcd}
					(R_{S_1}(C_1) \boxtimes_{\Ecal} R_{S_2}(C_2)) \boxtimes_{\Ecal} R_{S_3}(C_3) \arrow{r}{\rho} \arrow{d}{a_2} & R_{S_1 \times S_2}(C_1 \boxtimes_{\Ccal} C_2) \boxtimes_{\Ecal} R_{S_3}(C_3) \arrow{r}{\rho} & R_{S_1 \times S_2 \times S_3}((C_1 \boxtimes_{\Ccal} C_2) \boxtimes_{\Ccal} C_3) \arrow{d}{a_1} \\
					R_{S_1}(C_1) \boxtimes_{\Ecal} (R_{S_2}(C_2) \boxtimes_{\Ecal} R_{S_3}(C_3)) \arrow{r}{\rho} & R_{S_1}(C_1) \boxtimes_{\Ecal} R_{S_2 \times S_2}(C_2 \boxtimes_{\Ccal} C_3) \arrow{r}{\rho} & R_{S_1 \times S_2 \times S_3}(C_1 \boxtimes_{\Ccal} (C_2 \boxtimes_{\Ccal} C_3))
				\end{tikzcd}
			\end{equation*}
			is commutative.
		\end{enumerate}
		\item Let $c_{\Ccal}$ and $c_{\Ecal}$ be external commutativity constraints on $(\boxtimes_{\Ccal},m_{\Ccal})$ and $(\boxtimes_{\Ecal},m_{\Ecal})$, respectively. We say that $\rho$ is \textit{symmetric} (with respect to $c_{\Ccal}$ and $c_{\Ecal}$) if it satisfies the following additional condition:
		\begin{enumerate}
			\item[(mor-$c$ETS)] For every $S_1, S_2 \in \Scal$, the diagram of functors $\Ccal_1(S_1) \times \Ccal_1(S_2) \rightarrow \Ccal_2(S_1 \times S_2)$
			\begin{equation*}
				\begin{tikzcd}
					R_{S_1}(C_1) \boxtimes_{\Ecal} R_{S_2}(C_2) \arrow{rr}{\rho} \arrow{d}{c_{\Ecal}} && R_{S_1 \times S_2}(C_1 \boxtimes_{\Ccal} C_2) \arrow{d}{c_{\Ccal}} \\
					\tau^* (R_{S_2}(C_2) \boxtimes_{\Ecal} R_{S_1}(C_1)) \arrow{r}{\rho} & \tau^* R_{S_2 \times S_1}(C_2 \boxtimes_{\Ccal} C_1) \arrow{r}{\theta} & R_{S_1 \times S_2} (\tau^* (C_2 \boxtimes_{\Ccal} C_1))
				\end{tikzcd}
			\end{equation*}
		    (where $\tau: S_1 \times S_2 \xrightarrow{\sim} S_2 \times S_1$ denotes the permutation isomorphism) is commutative.
		\end{enumerate}
	\end{enumerate}
\end{defn}

With these notions at hand, we can define external associativity and commutativity constraints on additive $\Scal$-fibered representations in a natural way as follows:

\begin{defn}\label{defn:ETS-mor-asso-symm}
	Let $(\Ccal,\Ecal;\beta)$ be an (additive) $\Scal$-fibered representation endowed with an (additive) external tensor structure $(\boxtimes_{\Ccal},\boxtimes_{\Ecal};\delta)$. 
	\begin{enumerate}
		\item An \textit{external associativity constraint} $(a_{\Ccal},a_{\Ecal})$ on $(\boxtimes_{\Ccal},\boxtimes_{\Ecal};\delta)$ is the datum of
		\begin{itemize}
			\item an external associativity constraint $a_{\Ccal}$ on $(\boxtimes_{\Ccal},m_{\Ccal})$,
			\item an external associativity constraint $a_{\Ecal}$ on $(\boxtimes_{\Ecal},m_{\Ecal})$
		\end{itemize}
	    such that the external tensor structure $\nu$ on $\beta$ is associative in the sense of \Cref{defn:mor-acETS}(1).
	    \item An \textit{external commutativity constraint} $(c_{\Ccal},c_{\Ecal})$ on $(\boxtimes_{\Ccal},\boxtimes_{\Ecal};\delta)$ is the datum of
	    \begin{itemize}
	    	\item an external commutativity constraint $c_{\Ccal}$ on $(\boxtimes_{\Ccal},m_{\Ccal})$,
	    	\item an external commutativity constraint $c_{\Ecal}$ on $(\boxtimes_{\Ecal},m_{\Ecal})$
	    \end{itemize}
	    such that the external tensor structure $\nu$ on $\beta$ is symmetric in the sense of \Cref{defn:mor-acETS}(2). 
	\end{enumerate}
\end{defn}

We now show that external associativity and commutativity constraints of abelian $\Scal$-fibered representations lift to universal abelian factorizations in the expected way:

\begin{lem}\label{lem:lift-asso+symm}
	Let $(\Dcal,\Acal;\beta)$ be an abelian $\Scal$-fibered representation endowed with an exact external tensor structure $(\boxtimes_{\Dcal},\boxtimes_{\Acal};\nu)$; let $(\overline{\boxtimes},\bar{m})$ be the exact external tensor structure on $\A(\beta)$ obtained via \Cref{prop:lift-ETS}. Then the following statements hold:
	\begin{enumerate}
		\item Suppose that we are given an external associativity constraint $(a_{\Ccal},a_{\Ecal})$ on $(\boxtimes_{\Dcal},\boxtimes_{\Acal};\delta)$. Then, associating
		\begin{itemize}
			\item to every $S_1, S_2, S_3 \in \Scal$, the natural isomorphisms of functors $\A(\beta_{S_1}) \times \A(\beta_{S_2}) \times \A(\beta_{S_3}) \rightarrow \A(\beta_{S_1 \times S_2 \times S_3})$
			\begin{equation}\label{asso-bar}
				\overline{a}_{S_1,S_2,S_3}: (M_1 \overline{\boxtimes} M_2) \overline{\boxtimes} M_3 \xrightarrow{\sim} M_1 \overline{\boxtimes} (M_2 \overline{\boxtimes} M_3)
			\end{equation} 
		    obtained via \Cref{prop:lift-multinat}
		\end{itemize}
	    defines an external associativity constraint on $(\overline{\boxtimes},\bar{m})$, in such a way that the external tensor structures $(\boxtimes_{\Dcal},\overline{\boxtimes};\id)$ and $(\overline{\boxtimes},\boxtimes_{\Acal};\tilde{\nu})$ are associative.
	    \item Suppose that we are given an external commutativity constraint $(a_{\Ccal},a_{\Ecal})$ on $(\boxtimes_{\Dcal},\boxtimes_{\Acal};\delta)$. Then, associating
	    \begin{itemize}
	    	\item to every $S_1, S_2 \in \Scal$, the natural isomorphisms of functors $\A(\beta_{S_1}) \times \A(\beta_{S_2}) \rightarrow \A(\beta_{S_1 \times S_2})$
	    	\begin{equation}\label{symm-bar}
	    		\overline{c}_{S_1,S_2}: M_1 \overline{\boxtimes} M_2 \xrightarrow{\sim} \overline{\tau^*} (M_2 \overline{\boxtimes} M_1)
	    	\end{equation} 
    	    obtained via \Cref{prop:lift-multinat}
	    \end{itemize}
	    defines an external associativity constraint on $(\overline{\boxtimes},\bar{m})$, in such a way that the external tensor structures $(\boxtimes_{\Dcal},\overline{\boxtimes};\id)$ and $(\overline{\boxtimes},\boxtimes_{\Acal};\tilde{\nu})$ are symmetric.
	\end{enumerate}
\end{lem}
\begin{proof}
	\begin{enumerate}
		\item In order to show that the natural isomorphisms \eqref{asso-bar} define an external associativity constraint, we need to check that they satisfy condition ($a$ETS-1) and ($a$ETS-2) of \Cref{defn:ETS-asso-symm}(1). In view of \Cref{rem:lift-multinat}(2), this follows from the fact that the same condition is satisfied by the corresponding natural isomorphisms on the $\Scal$-fibered category $\Dcal$. Moreover, saying that the external tensor structures $\id$ on $\pi$ and $\tilde{\nu}$ on $\iota$ are associative is the same as saying that, for every $S_1, S_2, S_3 \in \Scal$, the pairs $(a_{\Dcal;S_1,S_2,S_3},\overline{a}_{S_1,S_2,S_3})$ and $(\overline{a}_{S_1,S_2,S_3},a_{\Acal;S_1,S_2,S_3})$ define $2$-morphisms of multi-linear representations, which is indeed the case by \Cref{prop:lift-multinat}.
		\item In order to show that the natural isomorphisms \eqref{symm-bar} define an external commutativity constraint, we need to show that they satisfy condition ($c$ETS) of \Cref{defn:ETS-asso-symm}(2). In view of \Cref{rem:lift-multinat}(2), this follows from the fact that the same condition is satisfied by the corresponding natural isomorphisms on the $\Scal$-fibered category $\Dcal$. Moreover, saying that the external tensor structures $\id$ on $\pi$ and $\tilde{\nu}$ on $\iota$ are associative is the same as saying that, for every $S_1, S_2 \in \Scal$, the pairs $(c_{\Dcal;S_1,S_2},\overline{c}_{S_1,S_2})$ and $(\overline{c}_{S_1,S_2},c_{\Acal;S_1,S_2})$ define $2$-morphisms of multi-linear representations, which is indeed the case by \Cref{prop:lift-multinat}.
	\end{enumerate}
\end{proof}

Finally, recall that in \cite[\S~6]{Ter23Fib} we explained how to translate the natural compatibility condition between associativity and commutativity in the setting of external tensor structures. We want to lift this to universal abelian factorizations as well. To this end, it suffices to recall one more definition:

\begin{defn}\label{defn:acETS}
	Let $\Ccal$ be an (additive) $\Scal$-fibered category endowed with an (additive) external tensor structure $(\boxtimes,m)$; in addition, let $a$ and $c$ be an external associativity constraint and an external commutativity constraint on $(\boxtimes,m)$, respectively.
	
	We say that the constraints $a$ and $c$ are \textit{compatible} if they satisfy the following condition:
	\begin{enumerate}
		\item[($ac$ETS)] For every $S_1, S_2, S_3 \in \Scal$, the diagram of functors $\Ccal(S_1) \times \Ccal(S_2) \times \Ccal(S_3) \rightarrow \Ccal(S_1 \times S_2 \times S_3)$
		\begin{equation*}
			\begin{tikzcd}[font=\small]
				(C_1 \boxtimes C_2) \boxtimes C_3 \arrow{r}{c} \arrow{dd}{a} & \tau_{(12)}^* (C_2 \boxtimes C_1) \boxtimes C_3 \arrow{r}{m} & \tau_{(12)(3)}^* ((C_2 \boxtimes C_1) \boxtimes C_3) \arrow{r}{a} & \tau_{(12)(3)}^* (C_2 \boxtimes (C_1 \boxtimes C_3)) \arrow{d}{c} \\
				&&& \tau_{(12)(3)}^* (C_2 \boxtimes \tau_{(13)}^* (C_3 \boxtimes C_1)) \arrow{d}{m} \\
				C_1 \boxtimes (C_2 \boxtimes C_3) \arrow{r}{c} & \tau_{(132)}^* ((C_2 \boxtimes C_3) \boxtimes C_1) \arrow{r}{a} & \tau_{(132)}^* (C_2 \boxtimes (C_3 \boxtimes C_1)) \arrow[equal]{r} & \tau_{(12)(3)}^* \tau_{(2)(13)}^* (C_2 \boxtimes (C_3 \boxtimes C_1))
			\end{tikzcd}
		\end{equation*}
		is commutative.
	\end{enumerate}
\end{defn}

\begin{lem}\label{lem:lift-ac}
	Let $(\Dcal,\Acal;\beta)$ be an abelian $\Scal$-fibered representation endowed with an exact external tensor structure $(\boxtimes_{\Dcal},\boxtimes_{\Acal};\nu)$; let $(\overline{\boxtimes},\bar{m})$ be the exact external tensor structure on $\A(\beta)$ obtained via \Cref{prop:lift-ETS}.
	
	Suppose that we are given an external associativity constraint $(a_{\Ccal},a_{\Ecal})$ and an external commutativity constraint $(c_{\Dcal},c_{\Acal})$ on $(\boxtimes_{\Dcal},\boxtimes_{\Acal};\delta)$ such that $a_{\Dcal}$ and $c_{\Dcal}$ are compatible. 
	Then the external associativity and commutativity constraints $\bar{a}$ and $\bar{c}$ on $(\overline{\boxtimes},\overline{m})$ obtained via \Cref{lem:lift-asso+symm} are compatible as well. 
\end{lem}
\begin{proof}
	We need to check that the external associativity constraint $\overline{a}$ and the external commutativity constraint $\overline{c}$ satisfy condition ($ac$ETS) of \Cref{defn:acETS}. In view of \Cref{rem:lift-multinat}(2), this follows from the fact that the same condition is satisfied by the corresponding natural isomorphisms on the $\Scal$-fibered category $\Dcal$.
\end{proof}

\begin{rem}
	In a similar way, it is possible to lift external tensor structures on morphisms of $\Scal$-fibered categories (see \Cref{defn:ETS-mor}) to the level of universal abelian factorizations; the same applies to the associativity and symmetry properties of tensor structures on morphisms (see \Cref{defn:ETS-mor-asso-symm}). However, we refrain from stating these results precisely since they do not seem to be relevant for our applications; the interested reader will find the correct statements without difficulty.
\end{rem}

\end{document}